\documentclass[times,sort&compress,3p,authoryear]{elsarticle}

\usepackage{amssymb}
\usepackage{amsmath}
\usepackage{amsthm}

\usepackage{hyperref}

\let\today\relax
\makeatletter
\def\ps@pprintTitle{%
    \let\@oddhead\@empty
    \let\@evenhead\@empty
    \def\@oddfoot{\footnotesize\itshape
         {} \hfill\today}%
    \let\@evenfoot\@oddfoot
    }
\makeatother

\newtheorem{theorem}{Theorem}
\newtheorem{lemma}{Lemma}

\newtheorem{corollary}{Corollary}

\usepackage{xcolor}

\journal{Journal of Multivariate Analysis}

\begin{document}

\begin{frontmatter}



\title{Asymptotic testing of covariance separability for matrix elliptical data} 


\author[1]{Joni Virta\corref{mycorrespondingauthor}}
\author[2,3]{Takeru Matsuda}

\address[1]{Department of Mathematics and Statistics, University of Turku}

\address[2]{Graduate School of Information Science and Technology, University of Tokyo}

\address[3]{Statistical Mathematics Unit, RIKEN Center for Brain Science}

\cortext[mycorrespondingauthor]{Corresponding author. Email address: \url{joni.virta@utu.fi}}

\begin{abstract}
We propose a new asymptotic test for the separability of a covariance matrix. The null distribution is valid in wide matrix elliptical model that includes, in particular, both matrix Gaussian and matrix $t$-distribution. The test is fast to compute and makes no assumptions about the component covariance matrices. An alternative, Wald-type version of the test is also proposed. Our simulations reveal that both versions of the test have good power even for heavier-tailed distributions and can compete with the Gaussian likelihood ratio test in the case of normal data.
\end{abstract}




\end{frontmatter}




\section{Introduction}\label{sec:intro}

Let $(\Omega, \mathcal{F}, \mathbb{P})$ be a probability space. Let $X_1, \ldots, X_n$ be a random sample of $p_1 \times p_2$ matrices, let $\mathrm{vec}(\cdot)$ denote the column-wise vectorization operator and let $\mathcal{S}_+^{p}$ be the set of $p \times p$ positive definite matrices. The analysis of matrix-valued data is commonly equipped with the assumption that the covariance matrix $\Sigma \in \mathcal{S}_+^{p_1 p_2}$ of $\mathrm{vec}(X)$ is \textit{separable}. That is, one assumes that $\Sigma = \Sigma_2 \otimes \Sigma_1$ for some positive definite matrices $\Sigma_1 \in \mathcal{S}_+^{p_1}$ and $\Sigma_2 \in \mathcal{S}_+^{p_2}$, where $\otimes$ is the Kronecker product. This structural assumption is natural in that $\Sigma_1$ and $\Sigma_2$ capture the dependency structures of the rows and columns of $X$, respectively.  

However, the separability assumption is also very strict, the representation $\Sigma_2 \otimes \Sigma_1$ admitting only $p_1(p_1 + 1)/2 + p_2(p_2 + 1)/2 - 1$ degrees of freedom versus the $p_1 p_2(p_1 p_2 + 1)/2$ in the fully unstructured covariance $\Sigma$. As such, before any analyses, it makes sense to carry out a test for the null hypothesis
\begin{align*}
    H_0: \mathrm{Cov} \{ \mathrm{vec}(X) \} = \Sigma_2 \otimes \Sigma_1 \quad \mbox{for some } \Sigma_1 \in \mathcal{S}_+^{p_1}, \Sigma_2 \in \mathcal{S}_+^{p_2}.
\end{align*}
The main contribution of the current work is to construct a test statistic for $H_0$ and derive its asymptotic null distribution. We work under the assumption that the sample $X_1, \ldots, X_n$ is drawn i.i.d. from the matrix-variate elliptical model $P(M, \Sigma_1, \Sigma_2, F)$ defined as follows. For fixed $p_1, p_2 \in \mathbb{Z}_+$, $M \in \mathbb{R}^{p_1 \times p_2}$ and $\Sigma_1 \in \mathcal{S}_+^{p_1}$, $\Sigma_2 \in \mathcal{S}_+^{p_2}$, we let $P(M, \Sigma_1, \Sigma_2, F)$ denote the probability distribution of the matrix
\begin{align}\label{eq:main_model}
    X := M + \Sigma_1^{1/2} Z \Sigma_2^{1/2},
\end{align}
where the $p_1 \times p_2$ random matrix $Z \sim F$ and $F$ is assumed to be continuous and matrix spherical, i.e., $O_1 Z O_2^{\top} \sim Z$ for all orthogonal matrices $O_1 \in \mathbb{R}^{p_1 \times p_1}$ and $O_2 \in \mathbb{R}^{p_2 \times p_2}$, see \cite{gupta2018matrix}. Further, we also assume that $F$ admits a finite fourth moment, that is, $\mathrm{E} \| Z \|_F^4 < \infty$, where $\| \cdot \|_F$ is the Frobenius norm. Theorem \ref{theo:fisher_consistency_1} later states that $H_0$ is indeed satisfied under \eqref{eq:main_model}. We also note that the scales of the parameters $\Sigma_1, \Sigma_2$ are confounded in the sense that $P(M, \Sigma_1, \Sigma_2, F)$ and $P(M, \lambda \Sigma_1, \lambda^{-1} \Sigma_2, F)$, $\lambda > 0$, are both the same model. The model $P(M, \Sigma_1, \Sigma_2, F)$ is a matrix-variate extension of the family of elliptical distributions commonly used in multivariate analysis \citep{fang1990symmetric}.

We base our testing of $H_0$ on estimating $\mathrm{Cov} \{ \mathrm{vec}(X) \}$ in both a separable and non-separable manner and comparing the two estimators in a suitable way. For the separable estimator we take the maximum likelihood estimator of the matrix normal distribution, which is defined as any matrices $S_{1n} \in \mathcal{S}_+^{p_1}$ and $S_{2n} \in \mathcal{S}_+^{p_2}$ such that
\begin{align}\label{eq:sample_structured}
\begin{split}
    S_{1n} =& \frac{1}{n p_2} \sum_{i = 1}^n (X_i - \Bar{X}) S_{2n}^{-1} (X_i - \Bar{X})^{\top},\\
    S_{2n} =& \frac{1}{n p_1} \sum_{i = 1}^n (X_i - \Bar{X})^{\top} S_{1n}^{-1} (X_i - \Bar{X}),
\end{split}
\end{align}
where $\Bar{X}=(1/n)\sum_{i=1}^n X_i$.
This estimator has been studied extensively in the context of matrix normal distribution, see, e.g., \cite{werner2008estimation, manceur2013maximum, drton2021existence, franks2021near, hoff2023core}, and we later show that $S_{1n}$ and $S_{2n}$ estimate $\Sigma_1$ and $\Sigma_2$ up to scale, respectively, also under the wider matrix elliptical model~\eqref{eq:main_model}. As with the covariance parameters, also the estimator pair $S_{1n}, S_{2n}$ is defined only up to scale in the sense that for any solution to \eqref{eq:sample_structured}, also the pair $\lambda S_{1n}, \lambda^{-1} S_{2n}$, $\lambda > 0$, is a solution. As such we standardize the solutions via determinants to $S_{1n}/|S_{1n}|^{1/p_1}$, $S_{2n}/|S_{2n}|^{1/p_2}$.

As our non-separable estimator we take the sample covariance matrix of the vectorized sample,
\begin{align*}
    S_n := \frac{1}{n} \sum_{i = 1}^n \mathrm{vec}(X_i - \Bar{X}) \mathrm{vec}(X_i - \Bar{X})^{\top},
\end{align*}
which completely ignores any possible structure in $\Sigma$, and which we also standardize by its determinant. Consequently, under the separable model \eqref{eq:main_model}, both of the matrices
\begin{align*}
    \frac{S_{2n}}{| S_{2n} |^{1/p_2}} \otimes \frac{S_{1n}}{| S_{1n} |^{1/p_1}} \quad \mbox{and} \quad \frac{S_{n}}{| S_{n} |^{1/(p_1 p_2)}}
\end{align*}
estimate the same quantity $\Theta := ( \Sigma_{2}/|\Sigma_{2}|^{1/p_2} ) \otimes ( \Sigma_{1}/|\Sigma_{1}|^{1/p_1} ) $. As such, a natural test statistic for the null hypothesis $H_0$ of separability is $ \| V_n - I_{p_1 p_2} \|_F^2 $, where
\begin{align}\label{eq:vn_definition}
     V_n := \left( \frac{S_n}{|S_n|^{1/(p_1 p_2)}} \right)^{-1/2} \left( \frac{S_{2n}}{| S_{2n} |^{1/p_2}} \otimes \frac{S_{1n}}{| S_{1n} |^{1/p_1}} \right) \left( \frac{S_n}{|S_n|^{1/(p_1 p_2)}} \right)^{-1/2},
\end{align}
and $\| \cdot \|_F$ denotes the Frobenius norm. In case $H_0$ is violated, the separable and non-separable estimator estimate different quantities and we expect $V_n$ to differ from $I_{p_1 p_2}$, making large values of the test statistic signify departures from $H_0$. The main theoretical result of this work concerns the limiting null distribution of this test statistic and is as follows.

\begin{theorem}\label{theo:main_limiting}
    Under $P(M, \Sigma_1, \Sigma_2, F)$, the limiting distribution of $n \| V_n - I_{p_1 p_2} \|_F^2$, as $n \rightarrow \infty$, is
    \begin{align*}
         \frac{ \mathrm{E}(z_{11}^2 z_{22}^2) + \mathrm{E}(z_{11}^2 z_{12}^2)}{\{ \mathrm{E}(z_{11}^2) \}^2} \chi^2_{\frac{1}{4}(p_1 + 2)(p_1 - 1)(p_2 + 2)(p_2 - 1)} + \frac{3 \mathrm{E}(z_{11}^2 z_{22}^2) - \mathrm{E}(z_{11}^2 z_{12}^2)}{\{ \mathrm{E}(z_{11}^2) \}^2} \chi^2_{\frac{1}{4} p_1 p_2 (p_1 - 1) (p_2 - 1)},
    \end{align*}
    where the two $\chi^2$-variates are independent.
\end{theorem}
The limiting distribution in Theorem \ref{theo:main_limiting} depends on the distribution $F$ only through the three moments, $\mathrm{E}(z_{11}^2 z_{22}^2)$, $\mathrm{E}(z_{11}^2 z_{12}^2)$ and $\mathrm{E}(z_{11}^2)$, and later in Section \ref{sec:h0_test} we show how to estimate these. In the special case of matrix normal distribution, for which the elements of $Z$ are i.i.d. standard normal, the two coefficients in Theorem \ref{theo:main_limiting} both equal 2 and the limiting distribution takes the simple form
\begin{align*}
    2 \chi^2_{\frac{1}{2} (p_1^2 - 1) (p_2^2 - 1) + \frac{1}{2} (p_1 - 1) (p_2 - 1)}.
\end{align*}

As an alternative approach, we consider also a Wald-type version of the test statistic, $n \mathrm{vec}(V_n - I_{p_1 p_2})' \hat{\Upsilon} \mathrm{vec}(V_n - I_{p_1 p_2})$, where $\hat{\Upsilon}$ is an estimate of the Moore-Penrose pseudoinverse of the asymptotic covariance matrix of $\sqrt{n} \mathrm{vec}(V_n - I_{p_1 p_2})$. When the data comes from the separable model $P(M, \Sigma_1, \Sigma_2, F)$, this statistic is shown to have a limiting chi-squared distribution with degrees of freedom $\frac{1}{2} (p_1^2 - 1) (p_2^2 - 1) + \frac{1}{2} (p_1 - 1) (p_2 - 1)$, i.e., the rank of the asymptotic covariance matrix, leading to a second novel test of separability.

Covariance separability testing has been considered in the earlier literature by the following works: \cite{lu2005likelihood, mitchell2005testing} proposed using a likelihood ratio test under the assumption of normality, as did \cite{roy2005implementation, simpson2010adjusted, filipiak2016score} but under the further assumptions that the covariance components were either compound symmetric or AR(1)-type; \cite{guggenberger2023test} based their test on the concept of closest separable covariance matrix \citep{van1993approximations} under the assumption that the data are rank-1 matrices; \cite{park2019permutation} divide the null hypothesis into subhypotheses concerning the separability of submatrices of $\Sigma$ and use a distribution-free permutation procedure to obtain the null distribution; \cite{kim2023robust} assumed the data to vector-elliptical and used the likelihood ratio test of the angular Gaussian distribution, simulating the null distribution with a permutation test; \cite{sung2025testing} considered a high-dimensional case where $p_1, p_2$ grow proportionally to $\sqrt{n}$ and based their test on the so-called Kroncker-core decomposition. To obtain the null distribution, they used either Monte Carlo simulation or an asymptotic distribution under the assumption that the elements of $Z$ are i.i.d.

Hence, compared to these previous works, our approach has the following advantages: (i) We make no restrictive assumptions about the structures of the covariance components $\Sigma_1, \Sigma_2$.  (ii) Our tests are asymptotic, meaning that running them is extremely fast, compared to computationally intensive permutation procedures. (iii) Our model $P(M, \Sigma_1, \Sigma_2, F)$ includes a large class of distributions, going significantly beyond Gaussianity and being also larger than the vector-elliptical family assumed in \cite{kim2023robust}. A matrix $X$ is said to be vector-elliptical if it obeys a variant of the model \eqref{eq:main_model} where $Z$ satisfies $\mathrm{vec}(Z) \sim O \mathrm{vec}(Z)$ for all $p_1 p_2 \times p_1 p_2$ orthogonal matrices $O$. The vector-elliptical model (which is sometimes also called matrix elliptical) is a strict subclass of our matrix elliptical model \eqref{eq:main_model}. For example, the matrix $t$-distribution \citep{gupta2018matrix} is a member of the latter family but does not belong to the former. Finally, we note that the standard matrix normal distribution is the only distribution that is simultaneously matrix spherical and has i.i.d. elements. Hence, our work can be seen as a complementary approach to \cite{sung2025testing}, who extended separability testing from Gaussianity to matrices with i.i.d. elements. 

The matrix elliptical family \eqref{eq:main_model} has been studied comparatively little in the literature, apart from the special case of the matrix normal distribution. \cite{gupta2013elliptically} provide a general introduction to matrix elliptical distributions and discuss the maximum likelihood estimation of the model parameters when the sample size is $n = 1$. \cite{diaz2021matrix, diaz2023singular} use matrix elliptical distributions to derive Birnbaum–Saunders distributions for random matrices. \cite{caro2016matrix, shinozaki2018distribution} derive the distribution of the largest eigenvalue of $X^{\top}X$ when $X \sim P(0, I_{p_1}, \Sigma_2, F)$. \cite{diaz2012statistical} use matrix elliptical distributions in shape analysis. \cite{van2016estimation} estimate the model parameters using a Bayesian approach. \cite{arashi2023some} consider maximum likelihood estimation under the tensor-variate counterpart of~\eqref{eq:main_model}. As a secondary contribution of this work, we provide novel results about the matrix elliptical model \eqref{eq:main_model}, in particular, the formula for the fourth moment matrix $\mathrm{E} [ \{ \mathrm{vec}(Z) \otimes \mathrm{vec}(Z) \} \{ \mathrm{vec}(Z) \otimes \mathrm{vec}(Z) \}^{\top} ]$ in Theorem \ref{theo:fourth_moment}. This result is of independent interest and valuable whenever one is interested in the asymptotic behavior of second-order methods for matrix-valued data.

This work is organized as follows. In Section \ref{sec:estimation} we derive asymptotic linearizations of the separable and non-separable covariance estimators under the matrix elliptical model \eqref{eq:main_model}. In Section \ref{sec:fourth_moment_Z} we compute the fourth moment matrix of the matrix spherical distribution $Z$ and evaluate it under a specific singular value decomposition based characterization of $Z$. In Section \ref{sec:h0_test} the previous results are combined to prove our main result Theorem \ref{theo:main_limiting} and, in the same section, we also formulate the alternative, Wald-type version of the test statistic along with its null distribution. Simulation results are given in Section \ref{sec:simulations}. Throughout the paper, we rely heavily on the formula $\mathrm{vec}(AXB^{\top}) = (B \otimes A) \mathrm{vec}(X)$, valid for all matrices $A, X, B$ of compatible sizes. Similarly, a major role is played by the commutation matrices $K_{m,n}$, defined as the $mn \times mn$ permutation matrices that satisfy
\begin{align*}
    K_{m,n} (A \otimes B) K_{s, t} = B \otimes A,
\end{align*}
for all matrices $A \in \mathbb{R}^{n \times s}$ and $B \in \mathbb{R}^{m \times t}$.

\section{Estimation theory}\label{sec:estimation}

\subsection{Population level}

We begin by formally showing that both estimators detailed earlier are Fisher consistent for $\Theta = ( \Sigma_{2}/|\Sigma_{2}|^{1/p_2} ) \otimes ( \Sigma_{1}/|\Sigma_{1}|^{1/p_1} )$ under the model \eqref{eq:main_model}. The population counterpart of the unstructured estimator is simply $S := \mathrm{E} ( [ \{ \mathrm{vec}(X) - \mathrm{vec}\{ \mathrm{E}(X) \} ] [ \{ \mathrm{vec}(X) - \mathrm{vec}\{ \mathrm{E}(X) \} ]^{\top} )$ and for the structured estimator one searches for $B_1 \in \mathcal{S}_+^{p_1}, B_2 \in \mathcal{S}_+^{p_2}$ satisfying the pair of equations

\begin{equation}\label{eq:population_structured}
    \begin{cases}
      B_1 &= \frac{1}{p_2} \mathrm{E} [ \{ X - \mathrm{E}(X) \} B_2^{-1} \{ X - \mathrm{E}(X) \} ^{\top} ] \\
      B_2 &= \frac{1}{p_1} \mathrm{E} [ \{ X - \mathrm{E}(X) \}^{\top} B_1^{-1} \{ X - \mathrm{E}(X) \} ].
    \end{cases}       
\end{equation}

Through this section, we use the notation $\beta_F := \mathrm{E}(Z_{11}^2)$.

\begin{theorem}\label{theo:fisher_consistency_1}
    For $X \sim P(M, \Sigma_1, \Sigma_2, F)$, we have
    \begin{align*}
        S = \beta_F (\Sigma_2 \otimes \Sigma_1).
    \end{align*}
\end{theorem}

\begin{theorem}\label{theo:fisher_consistency_2}
    For $X \sim P(M, \Sigma_1, \Sigma_2, F)$, the following hold.
    \begin{itemize}
        \item[(i)] $(B_1, B_2) = ( \beta_F^{1/2} \Sigma_1, \beta_F^{1/2} \Sigma_2)$ is a solution to \eqref{eq:population_structured}.
        \item[(ii)] If $(B_1, B_2)$ is a solution to \eqref{eq:population_structured}, then $B_1 = b_1 \Sigma_1$ and $B_2 = b_2 \Sigma_2$ for some $b_1, b_2 > 0$.
    \end{itemize}    
\end{theorem}

Theorems \ref{theo:fisher_consistency_1} and \ref{theo:fisher_consistency_2} imply that, after the scaling via determinants, Fisher consistent estimates of $\Theta$ are indeed obtained for both approaches. It can be shown that, in part (ii), the constants $b_1, b_2$ for any solution satisfy $b_1 b_2 = \beta_F$. We also note that, while the solution $(B_1, B_2)$ of \eqref{eq:population_structured} can thus be made unique by suitably constraining the scales of $B_1, B_2$, how this is done is not relevant due to our determinant scaling. For the sake of simplicity, we thus assume that $|B_1| = 1$ in \eqref{eq:population_structured} and that $|S_{1n}| = 1$ in the sample counterpart \eqref{eq:sample_structured}, again noting that this particular choice has no effect on our subsequent conclusions.

\subsection{Sample level}

Let next $X_1, \ldots, X_n$ be a random sample from the model $P(M, \Sigma_1, \Sigma_2, F)$. We begin by showing that both the unstructured and structured sample estimators converge in the almost sure sense. For the former, the claim follows directly from the classical strong law of large numbers and continuous mapping theorem, and we omit its proof. We use the notation $\Sigma := \Sigma_2 \otimes \Sigma_1$, implying that
\begin{align*}
    \frac{\Sigma}{| \Sigma |^{1/(p_1 p_2)}} = \frac{\Sigma_2}{| \Sigma_2 |^{1/p_2}} \otimes \frac{\Sigma_1}{| \Sigma_1 |^{1/p_1}}.
\end{align*}

\begin{theorem}
    For a random sample $X_1, \ldots, X_n$ from $P(M, \Sigma_1, \Sigma_2, F)$, we have
    \begin{align*}
        \frac{S_{n}}{| S_{n} |^{1/(p_1 p_2)}} \rightarrow_{a.s.} \frac{\Sigma}{| \Sigma |^{1/(p_1 p_2)}},
    \end{align*}
    as $n \rightarrow \infty$.
\end{theorem}

\begin{theorem}\label{theo:consistency_separable}
    For a random sample $X_1, \ldots, X_n$ from $P(M, \Sigma_1, \Sigma_2, F)$, we have
    \begin{align*}
        \frac{S_{1n}}{| S_{1n} |^{1/p_1}} \rightarrow_{a.s.} \frac{\Sigma_1}{| \Sigma_1 |^{1/p_1}} \quad \mbox{and} \quad \frac{S_{2n}}{| S_{2n} |^{1/p_2}} \rightarrow_{a.s.} \frac{\Sigma_2}{| \Sigma_2 |^{1/p_2}},
    \end{align*}
    as $n \rightarrow \infty$.
\end{theorem}

Having shown consistency, our next objective is to derive asymptotic linearizations for the sample counterparts of the estimators. But, before that, we present a general lemma for the first-order expansion of determinant-scaled estimators. Let $P_p := (1/p) \mathrm{vec}(I_p) \mathrm{vec}(I_p)^{\top}$ and $Q_p := I_{p^2} - P_p$. Namely, $Q_p$ denotes the projection matrix onto the orthogonal complement of the span of $\mathrm{vec}(I_p)$.

\begin{lemma}\label{lem:det_expansion}
    Let $A_n \in \mathcal{S}_+^p$ be a sequence of random matrices satisfying $\sqrt{n} (A_n - \gamma A) = \mathcal{O}_p(1)$ for some $\gamma > 0$ and $A \in \mathcal{S}_+^p$. Then,
    \begin{align*}
        \sqrt{n} \mathrm{vec} \left( \frac{A_n}{|A_n|^{1/p}} - \frac{A}{|A|^{1/p}} \right) = \frac{1}{\gamma |A|^{1/p}} (A^{1/2} \otimes A^{1/2}) Q_p (A^{-1/2} \otimes A^{-1/2}) \sqrt{n} \mathrm{vec} (A_n - \gamma A) + o_p(1). 
    \end{align*}
\end{lemma}

The next two theorems and corollary now give the asymptotic expansions of the determinant-scaled estimators in terms of the second moment vector of the data,
\begin{align*}
    f_n := \sqrt{n} \mathrm{vec} \left( \frac{1}{n} \sum_{i=1}^n \mathrm{vec}(Z_i) \mathrm{vec}(Z_i)^{\top} - \beta_F I_{p_1 p_2} \right) =\sqrt{n} \left[ \frac{1}{n} \sum_{i = 1}^n  \{ \mathrm{vec}(Z_i) \otimes \mathrm{vec}(Z_i) \} - \beta_F \mathrm{vec}(I_{p_1 p_2}) \right].
\end{align*}

\begin{theorem}\label{theo:bahadur_1}
    For a random sample $X_1, \ldots, X_n$ from $P(M, \Sigma_1, \Sigma_2, F)$, we have
    \begin{align*}
        \sqrt{n} \mathrm{vec} \left( \frac{S_n}{|S_n|^{1/(p_1 p_2)}} - \frac{\Sigma}{|\Sigma|^{1/(p_1 p_2)}} \right) = \frac{(\Sigma^{1/2} \otimes \Sigma^{1/2})}{\beta_F |\Sigma|^{1/(p_1 p_2)}} Q_{p_1 p_2} f_n + o_p(1).
    \end{align*}
\end{theorem}

\begin{theorem}\label{theo:bahadur_2}
    For a random sample $X_1, \ldots, X_n$ from $P(M, \Sigma_1, \Sigma_2, F)$, let $(S_{1n}, S_{2n})$ be a solution to \eqref{eq:sample_structured} under the constraint that $| S_{1n} | = 1$. Then, it holds that
    \begin{align*}
        \sqrt{n} \mathrm{vec} \left( \frac{S_{1n}}{|S_{1n}|^{1/p_1}} - \frac{\Sigma_1}{|\Sigma_1|^{1/p_1}} \right) = \frac{(\Sigma_1^{1/2} \otimes \Sigma_1^{1/2})}{\beta_F p_2 |\Sigma_1|^{1/p_1}}  Q_{p_1} \{ \mathrm{vec}(I_{p_2})^{\top} \otimes I_{p_1^2} \} (I_{p_2} \otimes K_{p_2, p_1} \otimes I_{p_1}) f_n + o_p(1),
    \end{align*}
    and
    \begin{align*}
        \sqrt{n} \mathrm{vec} \left( \frac{S_{2n}}{|S_{2n}|^{1/p_2}} - \frac{\Sigma_2}{|\Sigma_2|^{1/p_2}} \right)
        = \frac{(\Sigma_2^{1/2} \otimes \Sigma_2^{1/2})}{\beta_F p_1 |\Sigma_2|^{1/p_2}} Q_{p_2} \{ \mathrm{vec}(I_{p_1})^{\top} \otimes I_{p_2^2} \} (I_{p_1} \otimes K_{p_1, p_2} \otimes I_{p_2}) (K_{p_1, p_2} \otimes K_{p_1, p_2}) f_n + o_p(1).
    \end{align*}
\end{theorem}

We denote the two coefficient matrices in Theorem \ref{theo:bahadur_2} by
\begin{align}\label{eq:R1_R2_definition}
\begin{split}
    R_1 &:= \frac{1}{p_2} Q_{p_1} \{ \mathrm{vec}(I_{p_2})^{\top} \otimes I_{p_1^2} \} (I_{p_2} \otimes K_{p_2, p_1} \otimes I_{p_1}), \\
    R_2 &:= \frac{1}{p_1} Q_{p_2} \{ \mathrm{vec}(I_{p_1})^{\top} \otimes I_{p_2^2} \} (I_{p_1} \otimes K_{p_1, p_2} \otimes I_{p_2}) (K_{p_1, p_2} \otimes K_{p_1, p_2}).
\end{split}
\end{align}
The following result combines the two expansions in Theorem \ref{theo:bahadur_2} to obtain an analogous result for the Kronecker product estimator.

\begin{theorem}\label{theo:bahadur_4}
    Under the conditions of Theorem \ref{theo:bahadur_2}, we have
    \begin{align*}
        \sqrt{n} \mathrm{vec} \left( \frac{S_{2n}}{| S_{2n} |^{1/p_2}} \otimes \frac{S_{1n}}{| S_{1n} |^{1/p_1}} - \frac{\Sigma}{|\Sigma|^{1/(p_1 p_2)}} \right) = \frac{(\Sigma^{1/2} \otimes \Sigma^{1/2})}{\beta_F |\Sigma|^{1/(p_1 p_2)}} (I_{p_2} \otimes K_{p_1, p_2} \otimes I_{p_1}) \{ R_2 \otimes \mathrm{vec}(I_{p_1}) + \mathrm{vec}(I_{p_2}) \otimes R_1 \} f_n + o_p(1).
    \end{align*}
\end{theorem}

As a special case, Theorems \ref{theo:bahadur_2} and \ref{theo:bahadur_4} give the asymptotic linearization of the MLE of the matrix normal distribution into sample second moments under the assumption of Gaussian data. While the asymptotic distribution itself is well-known under Gaussianity \citep{werner2008estimation}, the linearization itself appears novel even under Gaussianity.

\section{Fourth moments of spherical matrices}\label{sec:fourth_moment_Z}

To summarize, the results in the previous section allow us to reduce the asymptotic behavior of the determinant-scaled estimators to the asymptotic behavior of the vector $f_n$. Being a second-order quantity, the limiting distribution of $f_n$ depends on the data distribution through its fourth moments, and, as such, in this section we first derive the general form of the fourth moment matrix of a matrix spherical distribution. Afterwards, we recall a characterization of matrix spherical distributions that allows parametrizing the fourth moment matrix via moments and cross-moments of the singular values of~$Z$.

All fourth moments of the matrix spherical $Z$ are collected in the matrix $A := \mathrm{E} [ \{ \mathrm{vec}(Z) \otimes \mathrm{vec}(Z) \} \{ \mathrm{vec}(Z) \otimes \mathrm{vec}(Z) \}^{\top} ]$. The double orthogonal invariance of $Z$ means that many elements of $A$ are zero (since each element $z_{jk}$ is symmetric w.r.t. zero) and many of the elements are equal to each other (since the rows of $Z$ are exchangeable, as are its columns). By enumerating the different possibilities, we observe that $A$ depends on the distribution of $Z$ through the five moments $m_1 := \mathrm{E}(z_{11}^4)$, $m_2 := \mathrm{E}(z_{11}^2 z_{12}^2)$, $m_3 := \mathrm{E}(z_{11}^2 z_{21}^2)$, $m_4 := \mathrm{E}(z_{11}^2 z_{22}^2)$ and $m_5 := \mathrm{E}(z_{11} z_{12} z_{21} z_{22})$. Furthermore, these five moments turn out to have only two degrees of freedom, as shown in the following result.

\begin{lemma}\label{lem:moment_connections}
    With the above notation, we have for matrix spherical $Z$ that
    \begin{align*}
        m_1 = 3 m_2, \quad m_2 = m_3, \quad 2m_5 = m_2 - m_4. 
    \end{align*}
\end{lemma}

To express the matrix $A$, we define the following ``building block'' matrices
\begin{align*}
    J_1 &:= \sum_{i, j = 1}^{p_1} (I_{p_2} \otimes e_i e_j^{\top} \otimes I_{p_2} \otimes e_i e_j^{\top}), \\
    J_2 &:= \sum_{i, j = 1}^{p_2} (e_i e_j^{\top} \otimes I_{p_1} \otimes e_i e_j^{\top} \otimes I_{p_1}), \\
    K^\otimes_1 &:= \sum_{i, j = 1}^{p_1} (I_{p_2} \otimes e_i e_j^{\top} \otimes I_{p_2} \otimes e_j e_i^{\top}), \\
    K^\otimes_2 &:= \sum_{i, j = 1}^{p_2} (e_i e_j^{\top} \otimes I_{p_1} \otimes e_j e_i^{\top} \otimes I_{p_1}),
\end{align*}
where $e_k$ denotes the $k$th standard basis vector and the corresponding space is evident from the indexing.

\begin{theorem}\label{theo:fourth_moment}
    For matrix spherical $Z$, we have 
\begin{align*}
    A =& \frac{1}{2} ( m_2 - m_4 ) (J_1 K^\otimes_2 + J_2 K^\otimes_1 + J_1 + J_2 + K^\otimes_1 + K^\otimes_2) + m_4 (I_{p_1^2 p_2^2} + J_1 J_2 + K^\otimes_1 K^\otimes_2).
\end{align*}
\end{theorem}



Next, we discuss a characterization of the family of all matrix spherical distributions. Assume, without loss of generality, that $p_1 \geq p_2$ and let $\mathcal{U}_{p_1 \times p_2}$ denote the probability distribution of a random matrix generated by taking the first $p_2$ columns of a Haar uniform $p_1 \times p_2$ orthogonal matrix. We consider the family of spherical $p_1 \times p_2$ matrices $Z$ defined as,
\begin{align}\label{eq:Z_spherical_family}
    Z = U \Lambda V^{\top},
\end{align}
where $U \sim \mathcal{U}_{p_1 \times p_2}$, $V \sim \mathcal{U}_{p_2 \times p_2}$, $\Lambda$ is a random diagonal matrix whose vector of positive diagonal elements $(\lambda_1, \ldots, \lambda_{p_2})$ has a distribution that is invariant to permutations, and $U, V, \Lambda$ are independent of each other. It is obvious that any such $Z$ has a matrix spherical distribution in the sense discussed in Section~\ref{sec:intro}. However, Theorem~9.3.9 in \cite{gupta2018matrix} states that also the reverse implication holds true, i.e., every matrix spherical distribution admits the representation \eqref{eq:Z_spherical_family}. Hence the family of $(p_1 \times p_2)$-dimensional matrix spherical distributions (with $p_1 \geq p_2$) has essentially $p_2$ ``degrees of freedom'', determined by the joint distribution of the $p_2$ singular values. This also means that we may express the moments of $Z$ using the moments of these singular values, as is done below.

\begin{theorem}\label{theo:Z_moments}
    For $Z$ as in \eqref{eq:Z_spherical_family}, we have
    \begin{align*}
        m_4 &= C^{-1} p_2 \mathrm{E}(\lambda_1^4) + C^{-1} \frac{p_2}{p_1 - 1} \left\{ (p_1 + 1)(p_2 + 1) +2 \right\} \mathrm{E}(\lambda_1^2 \lambda_2^2), \\
        m_2 - m_4 &= 2 C^{-1} p_2 \mathrm{E}(\lambda_1^4) - 2 C^{-1} p_2 \left( 1 + \frac{p_2 + 2}{p_1 - 1} \right) \mathrm{E}(\lambda_1^2 \lambda_2^2),
    \end{align*}
    where $C := p_1 (p_1 + 2) p_2 (p_2 + 2)$.
\end{theorem}


\section{Hypothesis tests for covariance separability}\label{sec:h0_test}

\subsection{Test based on $n \| V_n - I_{p_1 p_2} \|_F^2$}

We next propose a test for the null hypothesis
\begin{align*}
    H_0: \mbox{the covariance structure of $\mathrm{vec}(X)$ is separable.}
\end{align*}
Our test statistic of choice is $ t_n := n \| V_n - I_{p_1 p_2} \|_F^2 $, where the matrix
\begin{align}
     V_n := \left( \frac{S_n}{|S_n|^{1/(p_1 p_2)}} \right)^{-1/2} \left( \frac{S_{2n}}{| S_{2n} |^{1/p_2}} \otimes \frac{S_{1n}}{| S_{1n} |^{1/p_1}} \right) \left( \frac{S_n}{|S_n|^{1/(p_1 p_2)}} \right)^{-1/2}
\end{align}
is expected to be close to $I_{p_1 p_2}$ under separability, meaning that large values $t_n$ give evidence against $H_0$. As our main contribution of this section, we derive the exact limiting distribution of $t_n$ under the elliptical family \eqref{eq:main_model}, for which the null hypothesis of separable covariance structure holds. We begin with the following auxiliary, delta-method type result.

\begin{theorem}\label{theo:bahadur_5}
    Let $A_{1n}, A_{2n} \in \mathcal{S}_+^p$ be sequence of random matrices satisfying $\sqrt{n} (A_{kn} - A) = \mathcal{O}_p(1)$, $k = 1, 2$, for some $A \in \mathcal{S}_+^p$. Then,
    \begin{align*}
        \sqrt{n} \mathrm{vec}(A_{1n}^{-1/2} A_{2n} A_{1n}^{-1/2} - I_p) = (A^{-1/2} \otimes A^{-1/2}) \{ \sqrt{n} \mathrm{vec} ( A_{2n} - A ) - \sqrt{n} \mathrm{vec} ( A_{1n} - A ) \} + o_p(1).
    \end{align*}
\end{theorem}

Using now Theorems \ref{theo:bahadur_1}, \ref{theo:bahadur_4} and \ref{theo:bahadur_5}, we obtain the following corollary, where the asymptotic behavior of $\sqrt{n} \mathrm{vec}(V_n - I_{p_1 p_2})$ is reduced to that of the vectorized second moment $f_n$ defined in Section \ref{sec:estimation}. 

\begin{corollary}\label{cor:bahadur_6}
    Letting $V_n$ be as in \eqref{eq:vn_definition}, we have
    \begin{align*}
        \sqrt{n} \mathrm{vec}(V_n - I_{p_1 p_2}) = \frac{1}{\beta_F} [ (I_{p_2} \otimes K_{p_1, p_2} \otimes I_{p_1}) \{ R_2 \otimes \mathrm{vec}(I_{p_1}) + \mathrm{vec}(I_{p_2}) \otimes R_1 \} - Q_{p_1 p_2} ] f_n + o_p(1),
    \end{align*}
    where $R_1, R_2$ are as in \eqref{eq:R1_R2_definition}.
\end{corollary}

Combining the linearization in Corollary \ref{cor:bahadur_6} with the fourth moments derived in Theorem \ref{theo:fourth_moment} then leads, after tedious Kronecker product algebra to the following result, of which Theorem \ref{theo:main_limiting} in Section \ref{sec:intro} is then a restatement.

\begin{theorem}\label{theo:main_limiting_2}
    The test statistic $t_n$ satisfies
    \begin{align*}
        \beta_F^2 t_n \rightsquigarrow ( m_4 + m_2 ) \chi^2_{\frac{1}{4}(p_1 + 2)(p_1 - 1)(p_2 + 2)(p_2 - 1)} + ( 3 m_4 - m_2 ) \chi^2_{\frac{1}{4} p_1 p_2 (p_1 - 1) (p_2 - 1)},
    \end{align*}
    as $n \rightarrow \infty$.
\end{theorem}

Using the representation \eqref{eq:Z_spherical_family} and Theorem \ref{theo:Z_moments} it follows that the weight of the second $\chi^2$-distribution in Theorem~\ref{theo:main_limiting_2} can be written as
\begin{align*}
    3 m_4(Z) - m_2(Z) = \frac{2}{p_1 (p_1 - 1)} \mathrm{E}(\lambda_1^2 \lambda_2^2),
\end{align*}
showing that the weight is indeed always positive.

To use the test in practice, we still need to estimate $m_2/\beta_F^2$ and $m_4/\beta_F^2$. The following lemma shows that this can be achieved using the fourth moments of the standardized observations.

\begin{lemma}\label{lem:m2_and_m4}
    For matrix spherical $Z$, we have
    \begin{align*}
        \frac{1}{p_1 p_2}\mathrm{E}\| Z \|_F^2 &= \beta_F, \\
        \frac{1}{p_1 p_2}\mathrm{E}\| Z \|_F^4 &= m_2 (p_1 + p_2 + 1) + m_4 (p_1 - 1)(p_2 - 1), \\
        \frac{1}{p_1 p_2} \sum_{j = 1}^{p_1} \sum_{k = 1}^{p_2}  \mathrm{E} (z_{jk}^4) &= 3m_2.
    \end{align*}
\end{lemma}

Given a random sample $X_1, \ldots, X_n$ from our model, let now $Y_1, \ldots , Y_n$ denote the standardized sample $Y_i = S_{1n}^{-1/2} (X_i - \Bar{X}) S_{2n}^{-1/2}$. Furthermore, let
\begin{align*}
    d_{1n} := \frac{1}{p_1 p_2 n} \sum_{i = 1}^n \| Y_i \|_F^2, \quad d_{2n} := \frac{1}{p_1 p_2 n} \sum_{i = 1}^n \| Y_i \|_F^4, \quad d_{3n} := \frac{1}{p_1 p_2 n} \sum_{i = 1}^n \sum_{j = 1}^{p_1} \sum_{k = 1}^{p_2} y_{i, jk}^4, 
\end{align*}
using which we construct the estimators
\begin{align*}
    t_{1n} := \frac{d_{2n} + \frac{1}{3}(p_1 p_2 - 2 p_1 - 2 p_2) d_{3n}}{(p_1 - 1)(p_2 - 1) d_{1n}^2}, \quad t_{2n} := \frac{3 d_{2n} - \frac{1}{3}(p_1 + 2)(p_2 + 2) d_{3n} }{(p_1 - 1)(p_2 - 1) d_{1n}^2}.
\end{align*}
By Lemma \ref{lem:m2_and_m4} and the law of large numbers, these estimators satisfy $t_{1n} \rightarrow_P (m_4 + m_2)/\beta_F^2$ and $t_{2n} \rightarrow_P (3 m_4 - m_2)/\beta_F^2$, as $n \rightarrow \infty$, allowing us to approximate the limiting distribution of $t_n$ in Theorem \ref{theo:main_limiting_2}.

\subsection{Wald-type version of the test}

By Corollary \ref{cor:bahadur_6}, the vector $\sqrt{n} \mathrm{vec}(V_n - I_{p_1 p_2})$ has a limiting normal distribution with some asymptotic covariance matrix $\Psi$. Due to the symmetry and invariance properties of $V_n$, the matrix $\Psi$ is singular and, thus, as an alternative version of the test, we consider using the test statistic
\begin{align*}
    w_n = n \mathrm{vec}(V_n - I_{p_1 p_2})' \hat{\Upsilon} \mathrm{vec}(V_n - I_{p_1 p_2}),
\end{align*}
where the matrix $\hat{\Upsilon} \in \mathbb{R}^{p_1^2 p_2^2 \times p_1^2 p_2^2}$ is an estimator of the Moore-Penrose pseudoinverse $\Psi^{+}$, whose exact form is given below. Our next result shows that, as expected, the quantity $w_n$ has a limiting $\chi^2$-distribution with degrees of freedom equal to $\mathrm{rank}(\Psi) = \frac{1}{2} (p_1^2 - 1) (p_2^2 - 1) + \frac{1}{2} (p_1 - 1) (p_2 - 1)$.

\begin{theorem}\label{theo:extra_theorem}
    The test statistic $w_n$ satisfies
    \begin{align*}
        w_n \rightsquigarrow \chi^2_{\frac{1}{2} (p_1^2 - 1) (p_2^2 - 1) + \frac{1}{2} (p_1 - 1) (p_2 - 1)}
    \end{align*}
    as $n \rightarrow \infty$, where
    \begin{align*}
        \hat{\Upsilon} = \frac{1}{t_{1n}} B_0 G_1 B_0' + \frac{1}{t_{2n}} B_0 G_2 B_0'
    \end{align*}
    and the matrices $B_0, G_1, G_2$ are defined in the proof of this theorem.
\end{theorem}

\section{Simulations}\label{sec:simulations}

We next perform a simulation experiment to evaluate the finite-sample performance of the proposed methods. As a benchmark, we took the Gaussian likelihood ratio test (LRT) \citep{lu2005likelihood, mitchell2005testing} which is the ``default'' test for separability under continuous data. \textsf{R}-implementations of all three tests can be found in \url{https://github.com/jmvirta/ellipticalSeparabilityTest/}, where the CDF of the sum of scaled chi-squares in Theorem \ref{theo:main_limiting_2} is approximated using the implementation of Imhof's method in the package \texttt{CompQuadForm} \citep{Rcompquadform}. For heavy-tailed data it may in some cases with lower $n$ happen that $t_{2n} < 0$ in which case we set $t_{2n}$ to zero (without compromising its consistency) and drop the $B_0 G_2 B_0'$-part from $\hat{\Upsilon}$ in Theorem~\ref{theo:extra_theorem}.

We generate samples of size $n \in \{100, 200, 400, 800, 1600, 3200\}$ from the matrix-variate $t$-distribution $T_{p_1 \times p_2}(0, I_{p_1}, I_{p_2}, \nu)$ using the degrees of freedom $\nu \in \{ 3, 5, 7, \infty \}$ (for $\nu = \infty$, the distribution equals the standard matrix-variate normal). Restricting to identity covariance matrices and zero mean is without loss of generality since the tests are affine invariant in the sense that both samples $X_1, \ldots, X_n$ and $M + A_1 X_1 A_2^{\top}, \ldots, M + A_1 X_n A_2^{\top}$ give identical test statistic values. We consider two choices for the dimensions $(p_1, p_2) \in \{ (3, 3), (5, 5) \}$. Additionally, to create data for which the null hypothesis of separability fails, we multiply the $(1, 1)$-elements of the sample matrices by $(1 + \tau/\sqrt{n})$ where $\tau \in \{ 0, 1, 2, 3, 4, 5 \}$, creating a sequence of increasingly deviating local alternatives.

\begin{figure}[ht]
    \centering
    \includegraphics[width=0.85\linewidth]{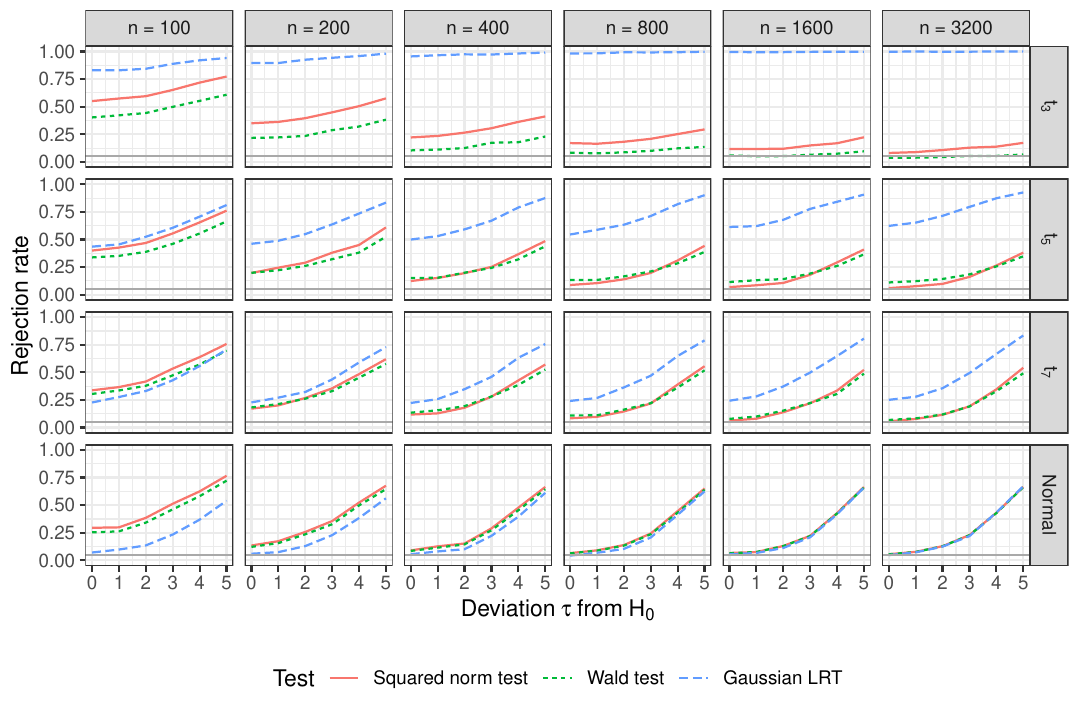}
    \caption{Average rejection rates (over 2000 replicates) of our proposed test and the Gaussian LRT for varying matrix $t$-distributions, sample sizes $n$ and local deviations $\tau$ from the null hypothesis when $(p_1, p_2) = (3, 3)$.}
    \label{fig:local_power_3x3}
\end{figure}

\begin{figure}[ht]
    \centering
    \includegraphics[width=0.85\linewidth]{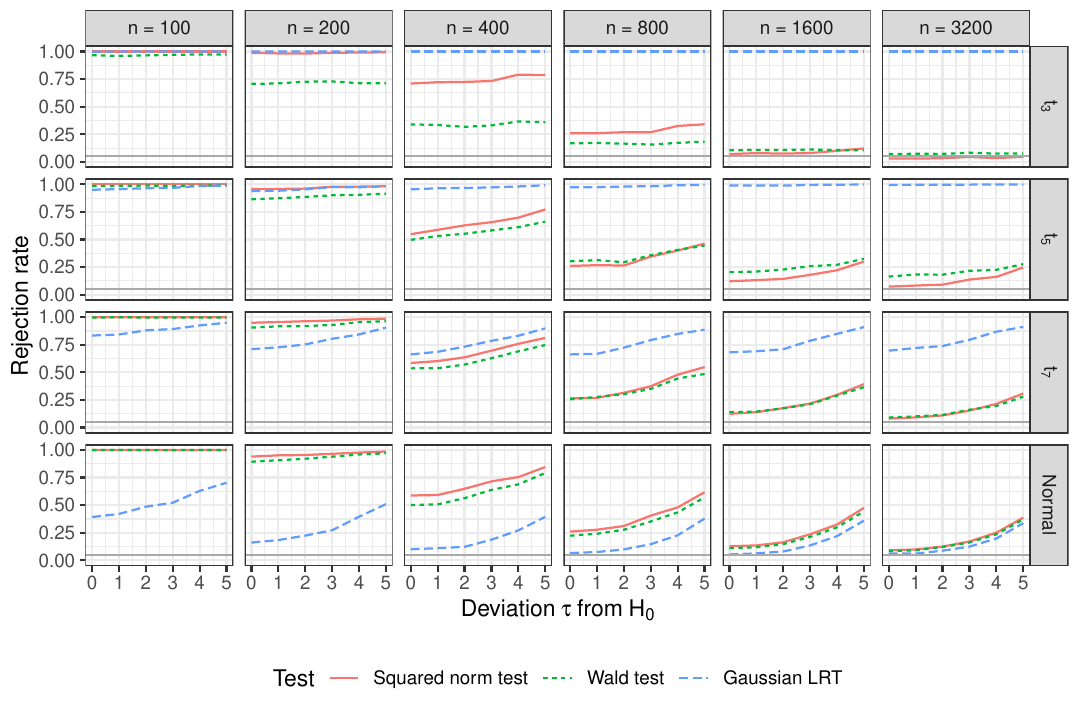}
    \caption{Average rejection rates (over 2000 replicates) of our proposed test and the Gaussian LRT for varying matrix $t$-distributions, sample sizes $n$ and local deviations $\tau$ from the null hypothesis when $(p_1, p_2) = (5, 5)$.}
    \label{fig:local_power_5x5}
\end{figure}

The simulation is repeated 2000 times and for each iteration we perform our proposed test (``Squared norm test'' in the figures), its Wald-type version and Gaussian LRT. The average proportions of rejections at the level $0.05$ are shown in Figures \ref{fig:local_power_3x3} and \ref{fig:local_power_5x5}. We make the following observations: (i) With increasing $n$, the squared norm test reaches the desired level at $\tau = 0$ in all cases except for $\nu = 3$ where the assumption of finite fourth moments does not hold. For the Wald test, slightly larger sample sizes are needed to reach the correct level. (ii) Greater sample sizes are required in the case $(p_1, p_2) = (5, 5)$ to reach the correct level. This is natural since an inevitable part of separability testing is the estimation of $\mathrm{Cov}\{ \mathrm{vec} (X_i) \}$ whose dimension quickly grows with $p_1, p_2$. (iii) The local power curves of our tests are increasing in $\tau$, as expected, and, after having reached the correct level, the squared norm test appears to have slightly better power than the Wald test. (iv) The Gaussian LRT retains the correct level only under normal data (bottom rows of panels), showing that it is not robust to violation of the Gaussianity assumption. (v) It appears that both our tests reach power equal to that of Gaussian LRT under normal data, as long as the sample size is large enough.


\subsection*{Acknowledgments}

JV was supported by the Research Council of Finland (grants 347501, 353769, 368494).
TM was supported by JSPS KAKENHI Grant Numbers 21H05205, 22K17865 and 24K02951 and JST Moonshot Grant Number JPMJMS2024.

\appendix

\section{Proofs of technical results}

Throughout the proofs, we make the assumption that the mean parameter equals zero, $M = 0$. Due to the centering employed by all our estimators, this assumption is without loss of generality.

\subsection{Proof for Section \ref{sec:estimation}}

\begin{proof}[Proof of Theorem \ref{theo:fisher_consistency_1}]
    As $\mathrm{vec}(X) = (\Sigma_2^{1/2} \otimes \Sigma_1^{1/2}) \mathrm{vec}(Z)$, it is sufficient to inspect $A := \mathrm{E}\{ \mathrm{vec}(Z) \mathrm{vec}(Z)^{\top} \}$, which satisfies $A = (O_2 \otimes O_1) A (O_2 \otimes O_1)^{\top}$ for all orthogonal matrices $O_1 \in \mathbb{R}^{p_1 \times p_1}$ and $O_2 \in \mathbb{R}^{p_2 \times p_2}$. By letting $O_1, O_2$ range over diagonal matrices with elements $\pm 1$ on the diagonal, we observe that $A$ is a diagonal matrix. Whereas, by ranging $O_1, O_2$ over permutation matrices, we further see that all the diagonal elements must be equal to each other. Hence $A = \mathrm{E}(Z_{11}^2) I_{p_1 p_2}$.
\end{proof}

\begin{proof}[Proof of Theorem \ref{theo:fisher_consistency_2}]
    We begin by showing that $\mathrm{E}(X X^{\top}) = \beta \mathrm{tr}(\Sigma_2) \Sigma_1$. To see this, we first note that $\mathrm{E}(X X^{\top}) = \Sigma_1^{1/2} \mathrm{E}(Z \Sigma_2  Z^{\top}) \Sigma_1^{1/2}$. Now, denoting by $\Lambda_2$ the diagonal matrix of the eigenvalues of $\Sigma_2$, the right-hand side orthogonal invariance of $Z$ shows that $\mathrm{E}(Z \Sigma_2  Z^{\top}) = \mathrm{E}(Z \Lambda_2 Z^{\top})$, whereas the left-hand side orthogonal invariance implies that $\mathrm{E}(Z \Lambda_2 Z^{\top})$ must be equal to $c I_{p_1}$ for some constant $c > 0$. By observing the $(1, 1)$ element and noting that each element of $Z$ has equal second moment, we get
    \begin{align*}
        c = \sum_{k = 1}^{p_2} \mathrm{E}(Z_{1k}^2) \lambda_{2k} = \mathrm{E}(Z_{11}^2) \mathrm{tr}(\Lambda_2),
    \end{align*}
    proving the desired formula $\mathrm{E}(X X^{\top}) = \beta \mathrm{tr}(\Sigma_2) \Sigma_1$. Using this it is now simple to check that claim (i) holds.

    For claim (ii), we have, using the same formula, that
    \begin{align*}
        B_1 = \frac{1}{p_2} \mathrm{E} (X B_2^{-1} X^{\top}) = \frac{1}{p_2} \beta \mathrm{tr}(\Sigma_2^{1/2} B_2^{-1} \Sigma_2^{1/2}) \Sigma_1,
    \end{align*}
    and the equivalent result for $B_2$, establishing also claim (ii).
\end{proof}

\begin{proof}[Proof of Theorem \ref{theo:consistency_separable}]
    Let $\mathcal{M}^p = \{ W \in \mathcal{S}_+^p \mid |W| = 1 \}$ denote the unit-determinant subset of $\mathcal{S}_+^p$. The Gaussian MLE $S_{1n}, S_{2n}$ is obtained as a minimizer of the map $(W_1, W_2) \mapsto L(W_1, W_2; X_i - \Bar{X})$ over $(W_1, W_2) \in \mathcal{M}^{p_1} \times \mathcal{S}_+^{p_2}$, where
    \begin{align*}
        L(W_1, W_2; X_i - \Bar{X}) \equiv L_n(W_1, W_2)
        := \sum_{i = 1}^n \mathrm{vec}(X_i - \Bar{X})^{\top} (W_2 \otimes W_1)^{-1} \mathrm{vec}(X_i - \Bar{X}) + \ln | W_2 \otimes W_1 |
    \end{align*}
    denotes the (shifted) negative log-likelihood function of the matrix normal distribution $\mathcal{N}_{p_1 \times p_2}(0, W_1, W_2)$ applied to the centered sample $X_1 - \Bar{X}, \ldots X_n - \Bar{X}$. By \cite[Theorem 4]{wiesel2012geodesic}, $L_n(W_1, W_2)$ is jointly $g$-convex in $W_1$ and $W_2$ and by \cite{duembgen2016geodesic}, the set $\mathcal{M}^{p_1} \times \mathcal{S}_+^{p_2}$ is $g$-convex. Theorem 1 in \cite{soloveychik2016gaussian} shows that, for large enough $n$, the continuity of the distribution of the $X_i$ guarantees that the function $(W_1, W_2) \mapsto L_n(W_1, W_2)$ almost surely admits a unique minimizer, i.e., the pair $S_{1n}, S_{2n}$. Moreover, it is easily seen via the law of large numbers that the function $L_n(W_1, W_2)$ converges point-wise in the almost sure sense, as $n \rightarrow \infty$, to a function $L(W_1, W_2)$ that is also jointly $g$-convex in $W_1$ and $W_2$ and whose unique minimizer is, by Theorem~\ref{theo:fisher_consistency_2}, $W_{01} := \Sigma_1/|\Sigma_1|^{1/p_1}$, $W_{02} := |\Sigma_1|^{1/p_1} \beta_F \Sigma_2$.

     Since $L_n(W_1, W_2)$ is jointly $g$-convex in $W_1$ and $W_2$, it is also convex when treated as a function of $(W_1, W_2)$ on the product manifold $\mathcal{M}^{p_1} \times \mathcal{S}_+^{p_2}$. The geodesics in the product manifold are of the form $\gamma := (\gamma_1, \gamma_2)$ where $\gamma_1, \gamma_2$ are geodesics in the component spaces, see Problem 5-7 in \cite{lee2018introduction}. In fact, by \cite{duembgen2016geodesic}, these geodesics even admit a closed form in the current case. The geodesic from $(M_1, S_1) \in \mathcal{M}^{p_1} \times \mathcal{S}_+^{p_2}$ to $(M_2, S_2) \in \mathcal{M}^{p_1} \times \mathcal{S}_+^{p_2}$ can be parameterized as
    \begin{align*}
        \gamma(t) = (M_1^{1/2} (M_1^{-1/2} M_2 M_1^{-1/2})^t M_1^{1/2}, S_1^{1/2} (S_1^{-1/2} S_2 S_1^{-1/2})^t S_1^{1/2}).   
    \end{align*}
    Hence, the desired result follows now by applying Theorem \ref{theo:manifold_minimization} to $L_n$ when treated as a function on the product manifold $\mathcal{M}^{p_1} \times \mathcal{S}_+^{p_2}$.

\end{proof}

Theorem \ref{theo:manifold_minimization} is a generalization of \cite[Theorem 5]{brunel2023geodesically} with a largely similar proof.

\begin{theorem}\label{theo:manifold_minimization}
    Let $(\mathcal{M}, g)$ be a complete Riemannian manifold. Let $f_n: \mathcal{M} \to \mathbb{R}$ be a sequence of $g$-convex random functions that converge point-wise in the almost sure sense to the non-random function $f: \mathcal{M} \to \mathbb{R}$, as $n \rightarrow \infty$. Assume further that $f_n$ admits the almost surely unique sequence of minimizers $Q_n \in \mathcal{M}$ and that $f$ has the unique minimizer $Q_0 \in \mathcal{M}$. Then, $d(Q_n, Q_0) \rightarrow_{a.s.} 0$, as $n \rightarrow \infty$. 
\end{theorem}

\begin{proof}[Proof of Theorem \ref{theo:manifold_minimization}]
    Fix $\varepsilon > 0$ and denote $K_\varepsilon := \{Q \in \mathcal{M} \mid d(Q, Q_0) = \varepsilon \}$. Then, since the metric $d$ is continuous, the bounded set $K_\varepsilon$ is closed, and hence also compact, $(\mathcal{M}, g)$ being a complete Riemannian manifold.     Consequently, arguing as in Lemmas 4 and 7 in \cite{brunel2023geodesically}, $f_n$ converges uniformly to $f$ on $K_\varepsilon$, in the almost sure sense. Moreover, $f$ attains its minimum value in $K_\varepsilon$ and, thus, $\eta := - \sup_{Q \in K_\varepsilon} \{ f(Q_0) - f(Q) \} > 0$.

    We next argue for an individual probability event $\omega \in \Omega$ for which
    \begin{align*}
        \sup_{Q \in K_\varepsilon} |f_n(Q) - f(Q)| \rightarrow 0 \quad \mbox{and} \quad | f_n(Q_0) - f(Q_0) | \rightarrow 0,
    \end{align*}
    as $n \rightarrow \infty$. We first show that for $n$ large enough, we have $f_n(Q) > f_n(Q_0)$ for all $Q \in K_\varepsilon$. To see this, we write, for arbitrary $Q \in K_\varepsilon$,
    \begin{align*}
        f_n(Q_0) - f_n(Q) = f_n(Q_0) - f(Q_0) + f(Q_0) - f(Q) + f(Q) - f_n(Q).
    \end{align*}
    Hence,
    \begin{align*}
        \sup_{Q \in K_\varepsilon} \{ f_n(Q_0) - f_n(Q) \} \leq |f_n(Q_0) - f(Q_0)| + \sup_{Q \in K_\varepsilon} \{ f(Q_0) - f(Q) \} + \sup_{Q \in K_\varepsilon} | f(Q) - f_n(Q) | \leq \frac{1}{3} \eta - \eta + \frac{1}{3} \eta, 
    \end{align*}
    for all $n$ large enough. Hence, $\inf_{Q \in K_\varepsilon} f_n(Q) > f_n(Q_0)$ for all $n$ large enough. In addition, we have $f_n(Q_n) \leq f_n(Q_0)$, by definition.

    Assume next that $d(Q_n, Q_0) > \varepsilon$ for infinitely many $n$. Then, for each such $n$, by the convexity of the manifold, we may choose $R_n \in K_\varepsilon$ such that $R_n = \lambda Q_0 + (1 - \lambda) Q_n$ (where the summation is an abuse of notation and means a geodesically convex combination of $Q_0$ and $Q_n$). Then, by the previous and the convexity of $f_n$, we get that $f_n(R_n) \leq \lambda f_n(Q_0) + (1 - \lambda) f_n(Q_n) < \inf_{Q \in K_\varepsilon} f_n(Q)$, which is a contradiction since $R_n \in K_\varepsilon$. Thus, for the probability element $\omega$ in question, we must have that $d(Q_n, Q_0) \leq \varepsilon$ for all large enough $n$. Since this holds for all $\omega$ in a set of measure one, the desired almost sure convergence follows.


    
\end{proof}

\begin{proof}[Proof of Lemma \ref{lem:det_expansion}]
    Using Slutsky's theorem we get
    \begin{align}\label{eq:det_expansion}
    \begin{split}
        \sqrt{n} \left( \frac{A_n}{|A_n|^{1/p}} - \frac{A}{|A|^{1/p}} \right)
        =& \sqrt{n} (A_n - \gamma A) \frac{1}{|A_n|^{1/p}} - \frac{1}{|A_n|^{1/p} |A|^{1/p}} A \sqrt{n} ( |A_n|^{1/p} - \gamma |A|^{1/p}) \\
        =& \sqrt{n} (A_n - \gamma A) \frac{1}{\gamma |A|^{1/p}} - \frac{1}{\gamma |A|^{2/p}} A \sqrt{n} ( |A_n|^{1/p} - \gamma |A|^{1/p}) + o_p(1).
    \end{split}    
    \end{align}
    We have the Taylor expansion $| I_p + \varepsilon G |^{1/p} = 1 + (1/p) \mathrm{tr}(G) \varepsilon + o(\varepsilon)$, implying the delta-method type result that
    \begin{align*}
        | A_n |^{1/p} =& |\gamma A + (A_n - \gamma A)|^{1/p} \\
        =& \gamma |A|^{1/p} | I_p + \gamma^{-1} A^{-1/2} (A_n - \gamma A) A^{-1/2}|^{1/p} \\
        =& \gamma |A|^{1/p} [ 1 + (1/p) \gamma^{-1} \mathrm{tr}\{ A^{-1} (A_n - \gamma A) \} ] + o_p(1/\sqrt{n}).
    \end{align*}
    Hence,
    \begin{align*}
        \sqrt{n} ( |A_n|^{1/p} - \gamma |A|^{1/p}) =& \frac{1}{p} |A|^{1/p} \mathrm{tr}\{ A^{-1/2} \sqrt{n} (A_n - \gamma A) A^{-1/2} \} + o_p(1) \\
        =& \frac{1}{p} |A|^{1/p} \mathrm{vec}(A^{-1})' \cdot \sqrt{n} \mathrm{vec} (A_n - \gamma A) + o_p(1),
    \end{align*}
    and denoting $f_n := \sqrt{n} \mathrm{vec} (A_n - \gamma A)$, we get from \eqref{eq:det_expansion} that
    \begin{align*}
        \sqrt{n} \mathrm{vec} \left( \frac{A_n}{|A_n|^{1/p}} - \frac{A}{|A|^{1/p}} \right)
        = \frac{1}{\gamma |A|^{1/p}} \left\{ I_p - \frac{1}{p} \mathrm{vec}(A) \mathrm{vec}(A^{-1})^{\top} \right\} f_n + o_p(1). 
    \end{align*}
    Observing then that $\mathrm{vec}(A) = (A^{1/2} \otimes A^{1/2}) \mathrm{vec}(I_p)$ and similarly for $\mathrm{vec}(A^{-1})$ gives the claim.
\end{proof}

\begin{proof}[Proof of Theorem \ref{theo:bahadur_1}]
    Denoting $p := p_1 p_2$, the standard CLT shows that $\sqrt{n} (S_n - \beta_F \Sigma) = \mathcal{O}_p(1)$, letting us use Lemma~\ref{lem:det_expansion} in conjunction with the facts that $S_n = \Sigma^{1/2} (1/n) \sum_{i = 1}^n \mathrm{vec}(Z_i) \mathrm{vec}(Z_i)^{\top} \Sigma^{1/2}$ and that $(\Sigma^{1/2} \otimes \Sigma^{1/2}) \mathrm{vec}(\Sigma^{-1}) = \mathrm{vec}(I_p)$ to obtain the desired result. 
\end{proof}

\begin{proof}[Proof of Theorem \ref{theo:bahadur_2}]
    In this proof, without loss of generality, we assume that the scaling of the solution $(S_{1n}, S_{2n})$ is fixed by taking $\mathrm{tr}(\Sigma^{-1/2}_1 S_{1n} \Sigma^{-1/2}_1) = 1$. By Theorem \ref{theo:consistency_separable} and the continuous mapping theorem, this rescaled solution is a consistent estimator of the corresponding population-level solution. We denote the resulting scaling constants for the solution $(B_1, B_2)$ to \eqref{eq:population_structured} by $b_1, b_2 > 0$, i.e., $B_1 = b_1 \Sigma_1$ and $B_2 = b_2 \Sigma_2$. Writing $\Tilde{X}_i := X_i - \Bar{X}$ and expanding the first equation in \eqref{eq:sample_structured} then gives
    \begin{align*}
        \sqrt{n} (S_{1n} - b_1 \Sigma_1) = \frac{1}{p_2 n} \sum_{i = 1}^n \Tilde{X}_i \sqrt{n} (S_{2n}^{-1} - b_2^{-1} \Sigma_2^{-1}) \Tilde{X}_i^{\top} + \sqrt{n} \left( \frac{1}{b_2 p_2 n} \sum_{i = 1}^n \Tilde{X}_i \Sigma_2^{-1} \Tilde{X}_i^{\top} - b_1 \Sigma_1 \right).
    \end{align*}
    The relation $B_n^{-1} - B^{-1} = B_n^{-1} (B - B_n) B^{-1}$ then yields
    \begin{align*}
        \sqrt{n} \mathrm{vec} (S_{1n} - b_1 \Sigma_1) = - \frac{1}{b_2 p_2 n} \sum_{i = 1}^n (\Tilde{X}_i \otimes \Tilde{X}_i) (\Sigma_2^{-1} \otimes S_{2n}^{-1}) \sqrt{n} \mathrm{vec} (S_{2n} - b_2 \Sigma_2) + \sqrt{n} \mathrm{vec} \left( \frac{1}{b_2 p_2 n} \sum_{i = 1}^n \Tilde{X}_i \Sigma_2^{-1} \Tilde{X}_i^{\top} - b_1 \Sigma_1 \right).
    \end{align*}
    Arguing as in Theorem \ref{theo:fisher_consistency_1}, we have $\mathrm{E} ((X-M) \otimes (X-M)) = \beta_F \mathrm{vec}(\Sigma_1) \mathrm{vec}(\Sigma_2)'$ for $X \sim P(M,\Sigma_1,\Sigma_2,F)$.
    Thus,
    \begin{align*}
        \frac{1}{n} \sum_{i = 1}^n (\Tilde{X}_i \otimes \Tilde{X}_i) (\Sigma_2^{-1} \otimes S_{2n}^{-1}) &= \beta_F \mathrm{vec}(\Sigma_1) \mathrm{vec}(\Sigma_2)' (\Sigma_2^{-1} \otimes S_{2n}^{-1}) + o_p(1) \\
        &= \beta_F \mathrm{vec}(\Sigma_1) \mathrm{vec}(S_{2n}^{-1})' + o_p(1) \\
        &= \frac{\beta_F}{b_2} \mathrm{vec}(\Sigma_1) \mathrm{vec}(\Sigma_2^{-1})' + o_p(1),
    \end{align*}
    giving us
    \begin{align*}
        \sqrt{n} \mathrm{vec} (S_{1n} - b_1 \Sigma_1)
        =& \left\{ -\frac{\beta_F}{b_2^2 p_2} \mathrm{vec}(\Sigma_1) \mathrm{vec}(\Sigma_2^{-1})' + o_p(1) \right\} \sqrt{n} \mathrm{vec} (S_{2n} - b_2 \Sigma_2) \\
        &+ (\Sigma_1^{1/2} \otimes \Sigma_1^{1/2}) \sqrt{n} \mathrm{vec} \left( \frac{1}{b_2 p_2 n} \sum_{i = 1}^n \Tilde{Z}_i \Tilde{Z}_i^{\top} - b_1 I_{p_1} \right) + o_p(1),
    \end{align*}
    where we have used both the law of large numbers and Theorem \ref{theo:consistency_separable}. Denoting the above expansion as $h_{1n} = \{ H_1 + o_p(1) \} h_{2n} + k_{1n} + o_p(1)$, by the symmetry of \eqref{eq:sample_structured} also $h_{2n} = \{ H_2 + o_p(1) \} h_{1n} + k_{2n} + o_p(1)$. Hence,
    \begin{align*}
        [ I_{p_1^2} -  \{ H_1 H_2 + o_p(1) \} ] h_{1n} = \{ H_1 + o_p(1) \} k_{2n} + k_{1n} + o_p(1),
    \end{align*}
    where $H_1 H_2 = p_1^{-1} \mathrm{vec}(\Sigma_1) \mathrm{vec}(\Sigma_1^{-1})^{\top}$. Since $\mathrm{vec}(\Sigma_1^{-1})^{\top} \mathrm{vec}(S_{1n}) = \mathrm{tr}(\Sigma^{-1/2}_1 S_{1n} \Sigma^{-1/2}_1)$, our choice of scaling shows that $H_1 H_2 h_{1n} = 0$ and, hence, Slutsky's theorem and the CLT give
    \begin{align}\label{eq:h1n_expansion}
        h_{1n} = H_1 k_{2n} + k_{1n} + o_p(1),
    \end{align}
    where the right-hand side has a limiting normal distribution. Thus, the conditions of Lemma \ref{lem:det_expansion} are satisfied for $\sqrt{n} (S_{1n} - b_1 \Sigma_1)$, and we next multiply \eqref{eq:h1n_expansion} from the left with
    $b_1^{-1} |\Sigma_1|^{-1/p} (\Sigma_1^{1/2} \otimes \Sigma_1^{1/2}) Q_{p_1} (\Sigma_1^{-1/2} \otimes \Sigma_1^{-1/2})$ to obtain

    \begin{align*}
        |\Sigma_1|^{1/p} \sqrt{n} \mathrm{vec} \left( \frac{S_{1n}}{|S_{1n}|^{1/p_1}} - \frac{\Sigma_1}{|\Sigma_1|^{1/p_1}} \right)
        =& - \frac{1}{p_2 \beta_F} (\Sigma_1^{1/2} \otimes \Sigma_1^{1/2}) Q_{p_1} \mathrm{vec}(I_{p_1}) \mathrm{vec}(I_{p_2})^{\top} \sqrt{n} \mathrm{vec} \left( \frac{1}{p_1 n} \sum_{i = 1}^n \Tilde{Z}_i^{\top} \Tilde{Z}_i - \beta_F I_{p_2} \right) \\
        &+ \frac{1}{\beta_F} (\Sigma_1^{1/2} \otimes \Sigma_1^{1/2}) Q_{p_1} \sqrt{n} \mathrm{vec} \left( \frac{1}{p_2 n} \sum_{i = 1}^n \Tilde{Z}_i \Tilde{Z}_i^{\top} - \beta_F I_{p_1} \right) + o_p(1),
    \end{align*}
    where we have used in particular $b_1 b_2 = \beta_F$ to simplify. Since $Q_{p_1} \mathrm{vec}(I_{p_1}) = 0$, the first term on the RHS above vanishes, and we then use the result
    \begin{align*}
        \sqrt{n} \mathrm{vec} \left( \frac{1}{p_2 n} \sum_{i = 1}^n \Tilde{Z}_i \Tilde{Z}_i^{\top} - \beta_F I_{p_1} \right) = \sqrt{n} \mathrm{vec} \left( \frac{1}{p_2 n} \sum_{i = 1}^n Z_i Z_i^{\top} - \beta_F I_{p_1} \right) + o_p(1),
    \end{align*}
    see page 28 in \cite{van2000asymptotic}, to obtain
    \begin{align*}
        \sqrt{n} \mathrm{vec} \left( \frac{S_{1n}}{|S_{1n}|^{1/p_1}} - \frac{\Sigma_1}{|\Sigma_1|^{1/p_1}} \right) = \frac{(\Sigma_1^{1/2} \otimes \Sigma_1^{1/2})}{\beta_F |\Sigma_1|^{1/p_1}}  Q_{p_1} \sqrt{n} \mathrm{vec} \left( \frac{1}{p_2 n} \sum_{i = 1}^n Z_i Z_i^{\top} - \beta_F I_{p_1} \right) + o_p(1).
    \end{align*}
    The desired claim for $S_{1n}$ now follows after using Lemmas \ref{lem:vec_aaprime} and \ref{lem:vec_aaprime_2} below. The equivalent result for $S_{2n}$ follows by symmetry and by noting that $\mathrm{vec}(A^{\top}) = K_{p_1, p_2} \mathrm{vec}(A)$ for a $p_1 \times p_2$ matrix $A$.
\end{proof}

We next present two lemmas concerning the properties of Kronecker products and commutation matrices.

\begin{lemma}\label{lem:vec_aaprime}
    For any $p_1 \times p_2$ matrix $A$, we have
    \begin{align*}
        \mathrm{vec}(AA^{\top}) = \{ \mathrm{vec}(I_{p_2})^{\top} \otimes I_{p_1^2} \} (I_{p_2} \otimes K_{p_2, p_1} \otimes I_{p_1}) \{ \mathrm{vec}(A) \otimes \mathrm{vec}(A) \}.
    \end{align*}
\end{lemma}

\begin{proof}[Proof of Lemma \ref{lem:vec_aaprime}]
    Denoting the columns of $A$ by $a_k$, we have $a_k = (e_k^{\top} \otimes I_{p_1}) \mathrm{vec}(A)$ where $e_k$ denotes the $k$th standard basis vector (of $\mathbb{R}^{p_2}$). Hence,
    \begin{align*}
        \mathrm{vec}(AA^{\top}) = \sum_{k = 1}^{p_2} a_k \otimes a_k = \left( \sum_{k = 1}^{p_2} e_k^{\top} \otimes I_{p_1} \otimes e_k^{\top} \otimes I_{p_1} \right) \{ \mathrm{vec}(A) \otimes \mathrm{vec}(A) \}.
    \end{align*}
    Since $(I_{p_1} \otimes e_k^{\top}) = (e_k^{\top} \otimes I_{p_1}) K_{p_2, p_1}$, the above simplifies to the desired form.
\end{proof}

\begin{lemma}\label{lem:vec_aaprime_2}
    We have
    \begin{align*}
        \{ \mathrm{vec}(I_{p_2})^{\top} \otimes I_{p_1^2} \} (I_{p_2} \otimes K_{p_2, p_1} \otimes I_{p_1}) \mathrm{vec}(I_{p_1 p_2}) = p_2 \mathrm{vec}(I_{p_1}). 
    \end{align*}
\end{lemma}

\begin{proof}[Proof of Lemma \ref{lem:vec_aaprime_2}]
    By writing the commutation matrix out, we get
    \begin{align*}
        (I_{p_2} \otimes K_{p_2, p_1} \otimes I_{p_1}) \mathrm{vec}(I_{p_1 p_2}) = \sum_{i = 1}^{p_2} \sum_{j = 1}^{p_1} \mathrm{vec} \{ (e_j e_i^{\top} \otimes I_{p_1}) (I_{p_2} \otimes e_j e_i^{\top}) \} = \sum_{i = 1}^{p_2} \sum_{j = 1}^{p_1} \mathrm{vec} \{ (e_j \otimes e_j) (e_i \otimes e_i)^{\top} \},
    \end{align*}
    from which the result follows after noting that $\mathrm{vec}(ab^{\top}) = b \otimes a$ and that $\sum_{j = 1}^{p_1} (e_j \otimes e_j) = \mathrm{vec} ( I_{p_1} )$.
\end{proof}

\begin{proof}[Proof of Theorem \ref{theo:bahadur_4}]
    Using Lemma 3 in \cite{magnus1985matrix} along with Slutsky's theorem, we obtain
    \begin{align*}
        & \sqrt{n} \mathrm{vec} \left( \frac{S_{2n}}{| S_{2n} |^{1/p_2}} \otimes \frac{S_{1n}}{| S_{1n} |^{1/p_1}} - \frac{\Sigma}{|\Sigma|^{1/(p_1 p_2)}} \right) \\
        =& (I_{p_2} \otimes K_{p_1, p_2} \otimes I_{p_1}) \sqrt{n} \left\{ \mathrm{vec} \left( \frac{S_{2n}}{| S_{2n} |^{1/p_2}} \right) \otimes \mathrm{vec} \left( \frac{S_{1n}}{| S_{1n} |^{1/p_1}} \right) \right. - \left. \mathrm{vec} \left( \frac{\Sigma_2}{|\Sigma_2|^{1/p_2}} \right) \otimes \mathrm{vec} \left( \frac{\Sigma_1}{|\Sigma_1|^{1/p_1}} \right) \right\}, \\
        =& (I_{p_2} \otimes K_{p_1, p_2} \otimes I_{p_1}) \left\{ \sqrt{n} \mathrm{vec} \left( \frac{S_{2n}}{|S_{2n}|^{1/p_2}} - \frac{\Sigma_2}{|\Sigma_2|^{1/p_2}} \right) \otimes \mathrm{vec} \left( \frac{\Sigma_1}{|\Sigma_1|^{1/p_1}} \right) \right\} \\
        +& (I_{p_2} \otimes K_{p_1, p_2} \otimes I_{p_1}) \left\{ \mathrm{vec} \left( \frac{\Sigma_2}{|\Sigma_2|^{1/p_2}} \right) \otimes \sqrt{n} \mathrm{vec} \left( \frac{S_{1n}}{|S_{1n}|^{1/p_1}} -  \frac{\Sigma_1}{|\Sigma_1|^{1/p_1}} \right) \right\} + o_p(1).
    \end{align*}
    Using Theorem \ref{theo:bahadur_2}, this implies that
    \begin{align*}
        \sqrt{n} \mathrm{vec} \left( \frac{S_{2n}}{| S_{2n} |^{1/p_2}} \otimes \frac{S_{1n}}{| S_{1n} |^{1/p_1}} - \frac{\Sigma}{|\Sigma|^{1/(p_1 p_2)}} \right) =& \frac{1}{\beta_F} (I_{p_2} \otimes K_{p_1, p_2} \otimes I_{p_1}) \left( \frac{\Sigma_2^{1/2} \otimes \Sigma_2^{1/2} \otimes \Sigma_1^{1/2} \otimes \Sigma_1^{1/2}}{|\Sigma_2|^{1/p_2}  |\Sigma_1|^{1/p_1}} \right) \\
        & \{ R_2 f_n \otimes \mathrm{vec}(I_{p_1}) + \mathrm{vec}(I_{p_2}) \otimes R_1 f_n \} + o_p(1),
    \end{align*}
    from which the result follows by noting that
    \begin{align*}
        R_2 f_n \otimes \mathrm{vec}(I_{p_1}) + \mathrm{vec}(I_{p_2}) \otimes R_1 f_n = \{ R_2 \otimes \mathrm{vec}(I_{p_1}) \} (f_n \otimes 1) + \{  \mathrm{vec}(I_{p_2}) \otimes R_1 \} (1 \otimes f_n).
    \end{align*}
    and that
    \begin{align*}
        (I_{p_2} \otimes K_{p_1, p_2} \otimes I_{p_1}) ( \Sigma_2^{1/2} \otimes \Sigma_2^{1/2} \otimes \Sigma_1^{1/2} \otimes \Sigma_1^{1/2} ) =&  ( \Sigma_2^{1/2} \otimes \Sigma_1^{1/2} \otimes \Sigma_2^{1/2} \otimes \Sigma_1^{1/2} )  (I_{p_2} \otimes K_{p_1, p_2} \otimes I_{p_1}) \\
         =& (\Sigma^{1/2} \otimes \Sigma^{1/2}) (I_{p_2} \otimes K_{p_1, p_2} \otimes I_{p_1}).
    \end{align*}
\end{proof}

\subsection{Proof for Section \ref{sec:fourth_moment_Z}}

\begin{proof}[Proof of Lemma \ref{lem:moment_connections}]
    Inspecting the top left $2 \times 2$ corner block of $Z$, we have, from the sphericity, that
    \begin{align}\label{eq:O_equivariance}
        \frac{1}{\sqrt{2}} \begin{pmatrix}
            1 & -1 \\
            1 & 1
        \end{pmatrix}
        \begin{pmatrix}
            z_{11} & z_{12} \\
            z_{21} & z_{22}
        \end{pmatrix} \sim
        \begin{pmatrix}
            z_{11} & z_{12} \\
            z_{21} & z_{22}
        \end{pmatrix}.
    \end{align}
    Computing now the fourth moments of the $(1, 1)$-elements of the above random matrices gives the first desired relation. The third relation is obtained by computing instead the expectations of the product of all elements of these matrices. For the second relation, the same strategy should be applied to the equivalent of \eqref{eq:O_equivariance} where the orthogonal transformation is applied from the right-hand side instead.
\end{proof}

\begin{proof}[Proof of Theorem \ref{theo:fourth_moment}]
    Below we denote $\Omega_0 := \mathrm{E} [ \{ \mathrm{vec}(Z) \otimes \mathrm{vec}(Z) \} \{ \mathrm{vec}(Z) \otimes \mathrm{vec}(Z) \}^{\top} ]$. Letting then $T_{ij}^{\top} := (e_j \otimes e_i \otimes I_{p_1 p_2})$, the matrices $T_{ij}^{\top} \Omega_0 T_{k\ell}$ correspond to all possible $p_1 p_2 \times p_1 p_2$ blocks of the matrix $\Omega_0$. We next verify the claim by inspecting each of these blocks in turn.

    Straightforward calculations using the formulas in the proof of Theorem \ref{theo:main_limiting} show that
    \begin{align}\label{eq:omega_blocks_1}
    \begin{split}
        T_{ij}^{\top} (I_{p_1^2 p_2^2} + J_1 J_2 + K^\otimes_1 K^\otimes_2) T_{k\ell} = \delta_{ik} \delta_{j\ell} I_{p_1^2 p_2^2} + (e_j \otimes e_i) (e_\ell \otimes e_k)^{\top} + (e_\ell \otimes e_k) (e_j \otimes e_i)^{\top}.
            \end{split}
    \end{align}
    Moreover, we also have
    \begin{align}\label{eq:omega_blocks_2}
    \begin{split}
        & T_{ij}^{\top} (J_1 K^\otimes_2 + J_2 K^\otimes_1 + J_1 + J_2 + K^\otimes_1 + K^\otimes_2) T_{k\ell} \\
        =& (e_\ell \otimes e_i)(e_j \otimes e_k)^{\top} + (e_j \otimes e_k)(e_\ell \otimes e_i)^{\top} \\
        +& \delta_{j \ell} (I_{p_2} \otimes e_i e_k^{\top} ) + \delta_{ik} (e_j e_\ell^{\top} \otimes I_{p_1}) + \delta_{j \ell} (I_{p_2} \otimes e_k e_i^{\top} ) + \delta_{ik} (e_\ell e_j^{\top} \otimes I_{p_1}).
        \end{split}
    \end{align}
    For the diagonal blocks, $i = k$, $j = \ell$, we thus have
    \begin{align}\label{eq:diagonal_blocks}
    \begin{split}
        T_{ij}^{\top} \Omega_0 T_{ij} = (m_2 - m_4) \{ (e_j \otimes e_i)(e_j \otimes e_i)^{\top} + (I_{p_2} \otimes e_i e_i^{\top}) + (e_j e_j^{\top} \otimes I_{p_1})\} + m_4 \{ I_{p_1^2 p_2^2} + 2 (e_j \otimes e_i) (e_j \otimes e_i)^{\top} \}.
        \end{split}
    \end{align}
    We compare this next to the form $T_{ij}^{\top} \Omega_0 T_{ij} = \mathrm{E} \{ Z_{ij}^2 \mathrm{vec}(Z) \mathrm{vec}(Z)^{\top} \}$. By the orthogonal invariance of $Z$, the latter matrix  
    must be diagonal. Indexing the elements of the matrix such that $(s, t) \otimes (s, t)$ refers to the diagonal element $(e_t \otimes e_s)^{\top} \mathrm{E} \{ Z_{ij}^2 \mathrm{vec}(Z) \mathrm{vec}(Z)^{\top} \} (e_t \otimes e_s)$, the diagonal elements are as follows: (i) the $(i, j) \otimes (i, j)$th diagonal element equals $m_1 = 3 m_2$, (ii) for $(s, t) \otimes (s, t)$ such that $s = i$ or $j = t$ (but not both), the corresdponding diagonal element is $m_2$, (iii) the remaining diagonal elements are $m_4$. Checking the diagonal elements of \eqref{eq:diagonal_blocks} shows that they match these and thus what remains is to verify the off-diagonal blocks.

    A general off-diagonal block satisfies $T_{ij}^{\top} \Omega_0 T_{ij} = \mathrm{E} \{ Z_{ij} Z_{k \ell} \mathrm{vec}(Z) \mathrm{vec}(Z)^{\top} \}$ and we first consider the case where $i = k$ and $j \neq \ell$. Using again the ``Kronecker-indexing'', this matrix has two kinds of non-zero elements. Its $(i, j) \otimes (i, \ell)$ and $(i, \ell) \otimes (i, j)$-elements are equal to $m_2$, whereas its $(r, j) \otimes (r, \ell)$ and $(r, \ell) \otimes (r, j)$-elements are equal to $0.5(m_2 - m_4)$, for all $r \neq i$. Again comparing to the corresponding elements of the claimed form of $\Omega_0$ using formulas \eqref{eq:omega_blocks_1} and \eqref{eq:omega_blocks_2}, we see that the claimed form has exactly this structure. The case where $i \neq k$ and $j = \ell$ behaves analogously and what remains is to check the case $i \neq k$, $j \neq \ell$. Here the block $\mathrm{E} \{ Z_{ij} Z_{k \ell} \mathrm{vec}(Z) \mathrm{vec}(Z)^{\top} \}$ has non-zero elements at (i) the locations $(i, j) \otimes (k, \ell)$ and $(k, \ell) \otimes (i, j)$ where the element is $m_4$, and at (ii) the locations $(i, \ell) \otimes (k, j)$ and $(k, j) \otimes (i, \ell)$ where the element is $0.5(m_2 - m_4)$. Comparison to the claimed form for $\Omega_0$ shows that it indeed has this structure, completing the proof.
\end{proof}




For computing the fourth moments of the spherical $Z$ in \eqref{eq:Z_spherical_family}, the following lemma first derives various moments of the Haar uniform matrix $U$.

\begin{lemma}\label{lem:U_moments}
    For $U \sim \mathcal{U}_{p_1 \times p_2}$, we have
    \begin{align*}
        \mathrm{E}(u_{11}^4) &= \frac{3}{p_1 (p_1 + 2)} \\
        \mathrm{E}(u_{11}^2 u_{12}^2) &= \frac{1}{p_1 (p_1 + 2)} \\
        \mathrm{E}(u_{11}^2 u_{21}^2) &= \frac{1}{p_1 (p_1 + 2)}
        \\
        \mathrm{E}(u_{11}^2 u_{22}^2) &= \frac{1}{p_1 (p_1 + 2)} \left( \frac{p_1 + 1}{p_1 - 1} \right)
    \end{align*}
\end{lemma}

\begin{proof}[Proof of Lemma \ref{lem:U_moments}]
    The first two equalities follow straight from \cite[Theorem 3.3]{fang1990symmetric} after observing that, prior to truncating the column dimension of $U$ to $p_2$, the marginal distributions of its rows have a uniform distribution on the unit sphere in $\mathbb{R}^{p_1}$. The third equality is proven similarly. For the fourth one, we first observe that
    \begin{align}\label{eq:U_haar_trace}
        p_2^2 = \mathrm{E} [ \{ \mathrm{tr}(UU^{\top}) \}^2 ] = \sum_{i = 1}^{p_1} \sum_{j = 1}^{p_2} \sum_{k = 1}^{p_1} \sum_{\ell = 1}^{p_2} \mathrm{E}(u_{ij}^2 u_{k\ell}^2).
    \end{align}
    Now, the restriction of this sum to the case $i = k$ gives
    \begin{align*}
        p_1 \sum_{j = 1}^{p_2} \sum_{\ell = 1}^{p_2} \mathrm{E}(u_{1j}^2 u_{1\ell}^2) =& p_1 p_2 \mathrm{E}(u_{11}^4) + p_1 p_2 (p_2 - 1) \mathrm{E}(u_{11}^2 u_{12}^2) = \frac{p_2 (p_2 + 2)}{(p_1 + 2)}
    \end{align*}
    where we have used the permutation invariance of $U$ and the first two parts of this lemma. Whereas for $i \neq k$ the sum \eqref{eq:U_haar_trace} reduces to
    \begin{align*}
        p_1 (p_1 - 1) \sum_{j = 1}^{p_2} \sum_{\ell = 1}^{p_2} \mathrm{E}(u_{1j}^2 u_{2\ell}^2) =& p_1 (p_1 - 1) p_2 \mathrm{E}(u_{11}^2 u_{21}^2) + p_1 (p_1 - 1) p_2 (p_2 - 1) \mathrm{E}(u_{11}^2 u_{22}^2) \\
        =& \frac{(p_1 - 1) p_2}{(p_1 + 2)} + p_1 (p_1 - 1) p_2 (p_2 - 1) \mathrm{E}(u_{11}^2 u_{22}^2)
    \end{align*}
    Combining the two cases with \eqref{eq:U_haar_trace} now gives the final claim of the lemma.
\end{proof}

\begin{proof}[Proof of Theorem \ref{theo:Z_moments}]
    Beginning with the first claim, we have
    \begin{align}\label{eq:Z_big_sum_1}
        \mathrm{E}(z_{11}^2 z_{22}^2) = \sum_{i,j,k,\ell = 1}^{p_2} \mathrm{E}(u_{1i} u_{1j} u_{2k} u_{2\ell}) \mathrm{E}(\lambda_i \lambda_j \lambda_k \lambda_\ell) \mathrm{E}(v_{1i} v_{1j} v_{2k} v_{2\ell}).
    \end{align}
    We break the sum \eqref{eq:Z_big_sum_1} into two parts, corresponding to the choices $i = j$ and $i \neq j$. In the former case, we must also have $k = \ell$, as otherwise the sign invariance of the columns of $U$ puts the corresponding expected value to zero. Hence, in the case $i = j$ we get
    \begin{align*}
        \sum_{i,k = 1}^{p_2} \mathrm{E}(u_{1i}^2 u_{2k}^2) \mathrm{E}(\lambda_i^2 \lambda_k^2 ) \mathrm{E}(v_{1i}^2 v_{2k}^2) =& p_2 \mathrm{E}(u_{11}^2 u_{21}^2) \mathrm{E}(\lambda_1^4) \mathrm{E}(v_{11}^2 v_{21}^2) + p_2 (p_2 - 1) \mathrm{E}(u_{11}^2 u_{22}^2) \mathrm{E}(\lambda_1^2 \lambda_2^2) \mathrm{E}(v_{11}^2 v_{22}^2)\\
        =&  \frac{p_2}{p_1 (p_1 + 2) p_2 (p_2 + 2)} \mathrm{E}(\lambda_1^4)\\
        +& \frac{p_2 (p_2 - 1)}{p_1 (p_1 + 2) p_2 (p_2 + 2)} \left(\frac{p_1 + 1}{p_1 - 1}\right) \left(\frac{p_2 + 1}{p_2 - 1}\right) \mathrm{E}(\lambda_1^2 \lambda_2^2),
    \end{align*}
    where the final step uses Lemma \ref{lem:U_moments}.

    For the case $i \neq j$ in \eqref{eq:Z_big_sum_1}, arguing as above shows that we get
    \begin{align*}
        2 p_2 (p_2 - 1) \mathrm{E}(u_{11} u_{12} u_{21} u_{22}) \mathrm{E}(v_{11} v_{12} v_{21} v_{22}) \mathrm{E}(\lambda_1^2 \lambda_2^2),
    \end{align*}
    where computation similar to the one used in Lemma \ref{lem:moment_connections} shows that
    \begin{align*}
        \mathrm{E}(u_{11} u_{12} u_{21} u_{22}) = \frac{1}{2}\left\{ \mathrm{E}(u_{11}^2 u_{12}^2) - \mathrm{E}(u_{11}^2 u_{22}^2) \right\} = -\frac{1}{p_1 (p_1 + 2)(p_1 - 1)}.
    \end{align*}
    Combining the previous hence shows the desired claim that
    \begin{align*}
        \mathrm{E}(z_{11}^2 z_{22}^2) &= \frac{p_2}{p_1 (p_1 + 2) p_2 (p_2 + 2)} \mathrm{E}(\lambda_1^4) + \frac{p_2 (p_2 - 1)}{p_1 (p_1 + 2) p_2 (p_2 + 2)} \left(\frac{p_1 + 1}{p_1 - 1}\right) \left(\frac{p_2 + 1}{p_2 - 1}\right) \mathrm{E}(\lambda_1^2 \lambda_2^2) \\
        &+ 2 p_2 (p_2 - 1) \frac{1}{p_1 (p_1 + 2)(p_1 - 1)} \frac{1}{p_2 (p_2 + 2)(p_2 - 1)} \mathrm{E}(\lambda_1^2 \lambda_2^2) \\
        &= C^{-1} p_2 \mathrm{E}(\lambda_1^4) + C^{-1} p_2 (p_2 - 1) \left\{ \left(\frac{p_1 + 1}{p_1 - 1}\right) \left(\frac{p_2 + 1}{p_2 - 1}\right) + 2 \frac{1}{(p_1 - 1)(p_2 - 1)} \right\} \mathrm{E}(\lambda_1^2 \lambda_2^2) \\
        &= C^{-1} p_2 \mathrm{E}(\lambda_1^4) + C^{-1} \frac{p_2}{p_1 - 1} \left\{ (p_1 + 1)(p_2 + 1) +2 \right\} \mathrm{E}(\lambda_1^2 \lambda_2^2).
    \end{align*}

    For the claim about $\mathrm{E}(z_{11}^2 z_{12}^2) - \mathrm{E}(z_{11}^2 z_{22}^2)$, we use a similar approach to compute $\mathrm{E}(z_{11}^2 z_{12}^2)$, considering the cases $i = j$ and $i \neq j$ separately. The former yields the contribution
    \begin{align*}
        p_2 \mathrm{E}(u_{11}^4) \mathrm{E}(v_{11}^2 v_{21}^2) \mathrm{E}(\lambda_1^4) + p_2 (p_2 - 1) \mathrm{E}(u_{11}^2 u_{12}^2) \mathrm{E}(v_{11}^2 v_{22}^2) \mathrm{E}(\lambda_1^2 \lambda_2^2) = 3 p_2 C^{-1} \mathrm{E}(\lambda_1^4) + p_2 (p_2 + 1) C^{-1} \mathrm{E}(\lambda_1^2 \lambda_2^2),
    \end{align*}
    and the latter gives
    \begin{align*}
        2 p_2 (p_2 - 1) \mathrm{E}(u_{11}^2 u_{12}^2) \mathrm{E}(v_{11} v_{12} v_{21} v_{22}) \mathrm{E}(\lambda_1^2 \lambda_2^2) = - 2 p_2 C^{-1} \mathrm{E}(\lambda_1^2 \lambda_2^2).
    \end{align*}
    Combining these gives
    \begin{align*}
        \mathrm{E}(z_{11}^2 z_{12}^2) &= 3 p_2 C^{-1} \mathrm{E}(\lambda_1^4) + p_2 (p_2 - 1) C^{-1} \mathrm{E}(\lambda_1^2 \lambda_2^2),
    \end{align*}
    from which the desired result now follows.
\end{proof}

\subsection{Proof for Section \ref{sec:h0_test}}

\begin{proof}[Proof of Theorem \ref{theo:bahadur_5}]
    To obtain the asymptotic behavior of $A_{1n}^{-1/2}$, we linearize the expression $\sqrt{n} ( A_{1n}^{-1/2} A_{1n} A_{1n}^{-1/2} - I_p ) = 0$ and use Slutsky's theorem to obtain
    \begin{align}\label{eq:bahadur_5_invsqrt}
    \begin{split}
        \sqrt{n} ( A_{1n}^{-1/2} - A^{-1/2}) A^{1/2} + A^{1/2} \sqrt{n} ( A_{1n}^{-1/2} - A^{-1/2} ) = - A^{-1/2} \sqrt{n} ( A_{1n} - A ) A^{-1/2} + o_p(1).
    \end{split}
    \end{align}
    Similarly, linearizing the LHS of the claim of the theorem, we obtain
    \begin{align*}
        \sqrt{n} (A_{1n}^{-1/2} A_{2n} A_{1n}^{-1/2} - I_p) =  \sqrt{n} (A_{1n}^{-1/2} - A^{-1/2}) A^{1/2} + A^{-1/2} \sqrt{n} ( A_{2n} - A) A^{-1/2} + A^{1/2} \sqrt{n} (A_{1n}^{-1/2} - A^{-1/2}) + o_p(1).
    \end{align*}
    Using \eqref{eq:bahadur_5_invsqrt} and vectorizing then yields
    \begin{align*}
        \sqrt{n} \mathrm{vec}(A_{1n}^{-1/2} A_{2n} A_{1n}^{-1/2} - I_p) = (A^{-1/2} \otimes A^{-1/2}) \{ \sqrt{n} \mathrm{vec} ( A_{2n} - A ) - \sqrt{n} \mathrm{vec} ( A_{1n} - A ) \} + o_p(1).
    \end{align*}
\end{proof}

In the proof of Theorem \ref{theo:main_limiting_2} we use the notation $L_1 := J_1/p_1, L_2 := J_2/p_2$ where $J_1, J_2$ are defined in Section \ref{sec:fourth_moment_Z}. The resulting matrices $L_1$ and $L_2$ are then orthogonal projection matrices and allow for simpler notational manipulation compared to $J_1$ and $J_2$.

\begin{proof}[Proof of Theorem \ref{theo:main_limiting_2}]
    By Corollary \ref{cor:bahadur_6} and Theorem \ref{theo:fourth_moment} it is sufficient to inspect the matrix $\Omega^{1/2} B \Omega^{1/2}$ where
    \begin{align*}
    B :=& [ (I_{p_2} \otimes K_{p_1, p_2} \otimes I_{p_1}) \{ R_2 \otimes \mathrm{vec}(I_{p_1}) + \mathrm{vec}(I_{p_2}) \otimes R_1 \} - Q_{p_1 p_2} ]^{\top} \\
    &[ (I_{p_2} \otimes K_{p_1, p_2} \otimes I_{p_1}) \{ R_2 \otimes \mathrm{vec}(I_{p_1}) + \mathrm{vec}(I_{p_2}) \otimes R_1 \} - Q_{p_1 p_2} ]
\end{align*}
and
\begin{align*}
    \Omega =& \frac{1}{2} \{ m_2(Z) - m_4(Z) \} (p_1 L_1 K^\otimes_2 + p_2 L_2 K^\otimes_1 + p_1 L_1 + p_2 L_2 + K^\otimes_1 + K^\otimes_2) \\
    &+ m_4(Z) (I_{p_1 p_2} + p_1 p_2 L_1 L_2 + K^\otimes_1 K^\otimes_2) - \beta_F^2 p_1 p_2 L_1 L_2,
\end{align*}
and where we have used the fact that $p_1 p_2 L_1 L_2 = \mathrm{vec}(I_{p_1 p_2}) \mathrm{vec}(I_{p_1 p_2})'$. The desired claim holds once we show that $\Omega^{1/2} B \Omega^{1/2}$ has exactly two non-zero eigenvalues with (a) multiplicities equal to the degrees of freedom of the postulated chi-squared distributions, and (b) magnitudes equal to the corresponding scalar factors.

Denote in the proof $K := K^\otimes_1 K^\otimes_2$, $L = L_1 L_2$. Then all of $L_1, L_2, L$ are orthogonal projection matrices and $K^\otimes_1, K^\otimes_2, K$ are involutions. The following proof relies heavily on the following connections between the different matrices: $K L = L K = L$, $L L_1 = L_1 L = L$, $L L_2 = L_2 L = L$,  $L K^\otimes_1 = K^\otimes_1 L = L$, $L K^\otimes_2 = K^\otimes_2 L = L$, $K^\otimes_1 K^\otimes_2 = K^\otimes_2 K^\otimes_1$, $K^\otimes_1 L_2 = L_2 K^\otimes_1$, $L_1 K^\otimes_2 = K^\otimes_2 L_1$, $L_1 L_2 = L_2 L_1 = P_{p_1 p_2}$, $K^\otimes_1 L_1 = L_1 K^\otimes_1 = L_1$, $K^\otimes_2 L_2 = L_2 K^\otimes_2 = L_2$.

Let next $H$ be an orthogonal projection matrix. Using the fact that for square matrices $A, B$, the eigenvalues of $A B$ and $B A$ are equal, we see that all the following matrices have equal spectra: $\Omega^{1/2} (H B H) \Omega^{1/2}$, $H B H \Omega$, $B H \Omega H$, $(H \Omega H)^{1/2} B (H \Omega H)^{1/2}$. Let now $H = (I_{p_1^2 p_2^2} - L_1) (I_{p_1^2 p_2^2} - L_2) (I_{p_1^2 p_2^2} - L_1 L_2) = I_{p_1^2 p_2^2} - L_1 - L_2 + L_1 L_2$, which is easily verified to be an orthogonal projection. This choice of $H$ satisfies $H B H = B$ (which we will show later), showing that it is sufficient to replace $\Omega$ with
\begin{align*}
    H \Omega H = \frac{1}{2} ( m_2 - m_4 ) H( K^\otimes_1 + K^\otimes_2 ) H + m_4 H (I_{p_1 p_2} + K^\otimes_1 K^\otimes_2) H,
\end{align*}
where we have used $H L_1 = H L_2 = H L = 0$. Moreover, we can absorb $H$ back to $B$ using the same reasoning above, meaning that, without loss of generality, we can take
\begin{align*}
    \Omega = \frac{1}{2} (m_2 - m_4) ( K^\otimes_1 + K^\otimes_2 ) + m_4 (I_{p_1 p_2} + K^\otimes_1 K^\otimes_2).
\end{align*}
We partition this matrix as
\begin{align*}
    \Omega = ( m_4 + m_2 ) \left[ \frac{1}{4} (I_{p_1^2 p_2^2} + K^\otimes_1 + K^\otimes_2 + K^\otimes_1 K^\otimes_2) \right] + ( 3 m_4 - m_2 ) \left[ \frac{1}{4} (I_{p_1^2 p_2^2} - K^\otimes_1 - K^\otimes_2 + K^\otimes_1 K^\otimes_2) \right],
\end{align*}
where the two matrices in square brackets, call them $G_1, G_2$, are orthogonal projection matrices that are mutually orthogonal, that is, $G_1 G_2 = 0$.

Next, we show that $G_1 B G_1 G_1 B G_1 = G_1 B G_1$, $G_2 B G_2 G_2 B G_2 = G_2 B G_2$, $G_1 B G_2 = G_2 B G_1 = 0$, which, together with the identity $G_1 B G_1 G_2 B G_2 = 0$, yield that
\begin{align*}
    \Omega^{1/2} B \Omega^{1/2} = ( m_4 + m_2 ) G_1 B G_1 + ( 3 m_4 - m_2 ) G_2 B G_2,
\end{align*}
where $G_1 B G_1$ and $G_2 B G_2$ are orthogonal projection matrices that are mutually orthogonal. Furthermore, the traces of these matrices are checked to be
\begin{align*}
    \mathrm{tr}(G_1 B G_1) =& \frac{1}{4}(p_1 + 2)(p_1 - 1)(p_2 + 2)(p_2 - 1) \\
    \mathrm{tr}(G_2 B G_2) =& \frac{1}{4} p_1 p_2 (p_1 - 1) (p_2 - 1),
\end{align*}
from which the claim then follows.


What remains is thus to establish the identities concerning $H, B, G_1, G_2$.  Recall that $R_1, R_2$ are defined as
\begin{align*}
\begin{split}
    R_1 &:= \frac{1}{p_2} Q_{p_1} \{ \mathrm{vec}(I_{p_2})^{\top} \otimes I_{p_1^2} \} (I_{p_2} \otimes K_{p_2, p_1} \otimes I_{p_1}), \\
    R_2 &:= \frac{1}{p_1} Q_{p_2} \{ \mathrm{vec}(I_{p_1})^{\top} \otimes I_{p_2^2} \} (I_{p_1} \otimes K_{p_1, p_2} \otimes I_{p_2}) (K_{p_1, p_2} \otimes K_{p_1, p_2}).
\end{split}
\end{align*}
It is straightforwardly shown that the matrix $U_1 \in \mathbb{R}^{p_1^2 \times p_1^2 p_2^2}$ defined as
\begin{align*}
    U_1^{\top} := \frac{1}{\sqrt{p_2}} \{ \mathrm{vec}(I_{p_2})^{\top} \otimes I_{p_1^2} \} (I_{p_2} \otimes K_{p_2, p_1} \otimes I_{p_1}) = \frac{1}{\sqrt{p_2}} \sum_{i = 1}^{p_2} \sum_{j = 1}^{p_1} e_i^{\top} \otimes e_j^{\top} \otimes e_j e_i^{\top} \otimes I_{p_1}
\end{align*}
has orthonormal columns, and the same holds for $U_2 \in \mathbb{R}^{p_2^2 \times p_1^2 p_2^2}$ defined as
\begin{align*}
    U_2^{\top} := \frac{1}{\sqrt{p_1}} \{ \mathrm{vec}(I_{p_1})^{\top} \otimes I_{p_2^2} \} (I_{p_1} \otimes K_{p_1, p_2} \otimes I_{p_2}) (K_{p_1, p_2} \otimes K_{p_1, p_2}).
\end{align*}
Consequently, $p_2 R_1^{\top} R_1  + p_1 R_2^{\top} R_2 = U_1 Q_{p_1} U_1^{\top} + U_2 Q_{p_2} U_2^{\top}$. Direct computation shows that
\begin{align*}
    (I_{p_2} \otimes K_{p_2, p_1} \otimes I_{p_1}) U_1 = \frac{1}{\sqrt{p_2}} \{ \mathrm{vec}(I_{p_2}) \otimes I_{p_1^2} \}
\end{align*}
and that
\begin{align*}
    (I_{p_2} \otimes K_{p_2, p_1} \otimes I_{p_1}) U_2 = \frac{1}{\sqrt{p_1}} \{ I_{p_2^2} \otimes \mathrm{vec}(I_{p_1}) \}.
\end{align*}
These together allow us to show that
\begin{align*}
    Q_{p_1 p_2}(I_{p_2} \otimes K_{p_1, p_2} \otimes I_{p_1}) \{ R_2 \otimes \mathrm{vec}(I_{p_1}) + \mathrm{vec}(I_{p_2}) \otimes R_1 \} =& Q_{p_1 p_2} ( U_1 Q_{p_1} U_1^{\top} + U_2 Q_{p_2} U_2^{\top}) \\
    =& U_1 Q_{p_1} U_1^{\top} + U_2 Q_{p_2} U_2^{\top},
\end{align*}
where the final step follows with direct computation using the definitions of $U_1$ and $U_2$. Consequently, we obtain that the matrix $B$ admits the form
\begin{align*}
    B = Q_{p_1 p_2} - U_1 Q_{p_1} U_1^{\top} - U_2 Q_{p_2} U_2^{\top},
\end{align*}
where direct computation reveals that $U_1 Q_{p_1} U_1^{\top} + U_2 Q_{p_2} U_2^{\top} = L_1 + L_2 - 2 P_{p_1 p_2}$, giving the alternate form
\begin{align}\label{eq:simplified_B}
    B = I_{p_1^2 p_2^2} - L_1 - L_2 + L_1 L_2.
\end{align}

The proof is now complete once we show that $G_1 B G_1 G_1 B G_1 = G_1 B G_1$, $G_2 B G_2 G_2 B G_2 = G_2 B G_2$, $G_1 B G_2 = G_2 B G_1 = 0$, $B L_1 = B L_2 = 0$ and compute the traces of $G_1 B G_1$ and $G_2 B G_2$. Matrix multiplication using the simplified form of $B$ and our earlier ``multiplication table'' gives both $G_1 B G_2 = G_2 B G_1 = 0$ and
\begin{align*}
    G_2 B G_2 = G_2 Q_{p_1 p_2} G_2 = G_2 (I_{p_1^2 p_2^2} - L) G_2 = G_2,   
\end{align*}
whose trace is easily seen to equal the claimed value. Next, we tackle the matrix $G_1 B G_1$, which is seen to be an orthogonal projection once we show that $B G_1 = B G_1 B G_1$. The matrix $B G_1$ can be seen to take the form
\begin{align*}
     \frac{1}{4} (I_{p_1^2 p_2^2} + K^\otimes_1 + K^\otimes_2 + K^\otimes_1 K^\otimes_2 - 2 L_1 - 2 L_2 - 2 K^\otimes_1 L_2 - 2 K^\otimes_2 L_1 + 4 L_1 L_2),
\end{align*}
from which the desired involution property can be verified. The same formula also allows us to verify that the trace $\mathrm{tr}(G_1 B G_1) = \mathrm{tr}(B G_1)$ equals $(p_1 + 2)(p_1 - 1)(p_2 + 2)(p_2 - 1)/4$. Finally, $B L_1 = B L_2 = 0$ follows immediately from the form~\eqref{eq:simplified_B}, concluding the proof.





\end{proof}

\begin{proof}[Proof of Lemma \ref{lem:m2_and_m4}]
Straightforward computation reveals that $\mathrm{E}\| Z \|_F^2 = p_1 p_2 \beta_F^2$. For the fourth moment $\mathrm{E}\| Z \|_F^4$, we first observe that $\mathrm{tr}(ZZ^{\top}) = h^{\top} \{ \mathrm{vec}^{\top}(Z) \otimes \mathrm{vec}^{\top}(Z) \}$ where
\begin{align*}
    h := \sum_{i = 1}^{p_1} \sum_{j = 1}^{p_2} \mathrm{vec}(e_i e_j^{\top}) \otimes \mathrm{vec}(e_i e_j^{\top}).
\end{align*}
Direct but tedious calculation then shows that $h^{\top} J_1 K^\otimes_2 h = h^{\top} J_1 h = p_1^2 p_2$, $h^{\top} J_2 K^\otimes_1 h = h^{\top} J_2 h = p_1 p_2^2$, $h^{\top} K^\otimes_1 h = h^{\top} K^\otimes_2 h = h^{\top}h = h^{\top} K^\otimes_1 K^\otimes_2 h = p_1 p_2$ and $ h^{\top} J_1 J_2 h = p_1^2 p_2^2$. Using Theorem \ref{theo:fourth_moment}, we thus obtain
\begin{align*}
    \frac{1}{p_1 p_2}\mathrm{E}\| Z \|_F^4 = \frac{1}{p_1 p_2} h^{\top} A h = ( m_2 - m_4 ) (p_1 + p_2 + 1) + m_4 (p_1 p_2 + 2) = m_2 (p_1 + p_2 + 1) + m_4 (p_1 - 1)(p_2 - 1).
\end{align*}
For the third claim, we note that, for all $j, k$ we have $\mathrm{E} (z_{jk}^4) = \mathrm{E} (z_{11}^4) = 3 m_2$ where the final equality uses Lemma~\ref{lem:moment_connections}.
\end{proof}

\begin{proof}[Proof of Theorem \ref{theo:extra_theorem}]
    By Corollary \ref{cor:bahadur_6} and the proof of Theorem \ref{theo:main_limiting_2}, the limiting covariance matrix of $\sqrt{n} \mathrm{vec}(V_n - I_{p_1 p_2})$ is $\Psi = \beta_F^{-2} B_0 \Omega B_0'$,
    where $\Omega$ is defined in the proof of Theorem \ref{theo:main_limiting_2} and
    \begin{align*}
        B_0 := (I_{p_2} \otimes K_{p_1, p_2} \otimes I_{p_1}) \{ R_2 \otimes \mathrm{vec}(I_{p_1}) + \mathrm{vec}(I_{p_2}) \otimes R_1 \} - Q_{p_1 p_2}.
    \end{align*}
    Defining next the symmetric matrix $H$ as in the proof of Theorem \ref{theo:main_limiting_2}, it can be shown with the techniques used in that proof that $H B_0' = B_0'$. Using this and the proof of Theorem \ref{theo:main_limiting_2}, we obtain that
    \begin{align*}
        \Psi = \frac{1}{\beta_F^{2}} B_0 H \Omega H B_0' = \frac{m_4 + m_2}{\beta_F^2} B_0 G_1 B_0' + \frac{3 m_4 - m_2}{\beta_F^2} B_0 G_2 B_0',
    \end{align*}
    where $G_1, G_2$ are as in the proof of Theorem \ref{theo:main_limiting_2}. Arguing then again as in the proof of Theorem \ref{theo:main_limiting_2}, we may show that $B_0 G_1 B_0'$ and $B_0 G_2 B_0'$ are orthogonal projection matrices that are also mutually orthogonal, $B_0 G_1 B_0' B_0 G_2 B_0' = 0$. Hence, the Moore-Penrose pseudoinverse of $\Psi$ is
    \begin{align*}
        \Psi^{+} = \frac{\beta_F^2}{m_4 + m_2} B_0 G_1 B_0' + \frac{\beta_F^2}{3 m_4 - m_2} B_0 G_2 B_0',
    \end{align*}
    and, by the discussion following Lemma \ref{lem:m2_and_m4}, the matrix $\hat{\Upsilon}$ is thus a consistent estimator of $\Psi^{+}$. From this it follows that the limiting distribution of $w_n$ is the same as that of $y' \Psi^{1/2} \Psi^{+} \Psi^{1/2} y$, where $y$ has a standard $p_1^2 p_2^2$-variate normal distribution, i.e., a $\chi^2$-distribution with degrees of freedom equal to $\mathrm{rank}(\Psi)$. That this equals
    \begin{align*}
        \frac{1}{4}(p_1 + 2)(p_1 - 1)(p_2 + 2)(p_2 - 1) + \frac{1}{4} p_1 p_2 (p_1 - 1) (p_2 - 1) = \frac{1}{2} (p_1^2 - 1) (p_2^2 - 1) + \frac{1}{2} (p_1 - 1) (p_2 - 1)
    \end{align*}
    can be shown as in the proof of Theorem \ref{theo:main_limiting_2}.

\end{proof}

\bibliographystyle{elsarticle-harv} 
\bibliography{references}



\end{document}